\documentclass{siamart171218}  

\usepackage{url}

\title{The Occupation Kernel Method for Nonlinear System Identification\thanks{This paper gives a considerable extension to a conference paper published in the proceedings of the 2019 IEEE Conference on Decision and Control \cite{SCC.Rosenfeld.Kamalapurkar.ea2019a}. This research was supported by the Air Force Office of Scientific Research (AFOSR) under contract numbers FA9550-20-1-0127, FA9550-16-1-0246, FA9550-18-1-0122, and FA9550-21-1-0134 the Office of Naval Research (ONR) under contract number N00014-18-1-2184, the Air Force Research Laboratories (AFRL) under contract number FA8651-19-2-0009 and the Defense Advanced Research Projects Agency (DARPA) under contract number FA8750-18-C-0089, and the National Science Foundation under award numbers 2027999, 2027976, and 2028001. Any opinions, findings and conclusions or recommendations expressed in this material are those of the author(s) and do not necessarily reflect the views of the sponsoring agencies. The authors would like to express our appreciation for Hongliang Wang for his careful reading of the manuscript and bringing to our attention numerous errors and typos. A part of this work was completed utilizing the High Performance Computing Center facilities of Oklahoma State University at Stillwater.} The following YouTube playlist discusses the contents of this manuscript: \url{https://youtube.com/playlist?list=PLldiDnQu2phtKtASlUrmiPBH705xmRdUC}}

\author{Joel A. Rosenfeld%
\thanks{Department of Mathematics and Statistics, University of South Florida, Tampa, FL 33620 USA}%
\and Benjamin Russo%
\thanks{Department of Mathematics, Farmingdale State College, Farmingdale, NY 11735 USA}%
\and Rushikesh Kamalapurkar%
\thanks{Department of Mechanical and Aerospace Engineering, Oklahoma State University, Stillwater, OK 74078 USA}%
\and Taylor T. Johnson\thanks{Institute for Software Integrated Systems, Department of Electrical Engineering and Computer Science, Vanderbilt University, Nashville, TN 37212 USA}}
\usepackage{savesym}

\usepackage{graphicx}
\usepackage{amssymb,amsfonts,amsmath}
\usepackage[caption=false]{subfig}
\usepackage[normalem]{ulem}
\useunder{\uline}{\ul}{}

\usepackage[normalem]{ulem}
\useunder{\uline}{\ul}{}
\usepackage{algorithmic}

\newcounter{assumptioncounter}
\newcounter{systemcounter}
\newcounter{expcounter}

\newtheorem{example}{Example}
\newtheorem{remark}{Remark}

\newtheorem{assumption}[assumptioncounter]{Assumption}
\newtheorem{system}[systemcounter]{System}
\newtheorem{experiment}[expcounter]{Experiment}

\DeclareMathOperator{\vspan}{span}
\DeclareMathOperator{\diag}{diag}

\usepackage{color}

\begin{document}

\maketitle

\begin{abstract}
This manuscript presents a novel approach to nonlinear system identification leveraging densely defined Liouville operators and a new ``kernel'' function that represents an integration functional over a reproducing kernel Hilbert space (RKHS) dubbed an occupation kernel. The manuscript thoroughly explores the concept of occupation kernels in the contexts of RKHSs of continuous functions, and establishes Liouville operators over RKHS where several dense domains are found for specific examples of this unbounded operator. The combination of these two concepts allow for the embedding of a dynamical system into a RKHS, where function theoretic tools may be leveraged for the examination of such systems. This framework allows for trajectories of a nonlinear dynamical system to be treated as a fundamental unit of data for nonlinear system identification routine. The approach to nonlinear system identification is demonstrated to identify parameters of a dynamical system accurately, while also exhibiting a certain robustness to noise.
\end{abstract}

\section{Introduction}

Consider a dynamical system $\dot x = f(x)$, where $x :[0,T] \to \mathbb{R}^n$ is the system state and $f :\mathbb{R}^n \to \mathbb{R}^n$ are Lipschitz continuous dynamics. Dynamical systems are prevalent in the sciences, such as engineering \cite{kamalapurkar2018reinforcement,dixon2013nonlinear,khalil2002nonlinear}, biology \cite{Bartlett1973,freedman1980deterministic}, neuroscience \cite{izhikevich2008dynamical}, physics \cite{stewart2000non}, and mathematics \cite{coddington1955theory,perko2013differential}. However, in many cases even physically motivated dynamical systems can have unknown parameters (i.e. a gray box model), such as mass and length of mechanical components, or the dynamics may be completely unknown (i.e. a black box model) \cite{nelles2013nonlinear}. In such cases, system identification methods are leveraged to estimate the dynamics of the system based on data generated by the system \cite{nelles2013nonlinear}.

While many classical tools are available for system identification for linear dynamics using the impulse response and Fourier and Laplace transforms of the dynamical system, the identification of nonlinear systems proves more challenging as nonlinearities may manifest in a variety of ways, and linear transform methods for general nonlinear systems are unavailable \cite{billings2013nonlinear,nelles2013nonlinear}.

To address these challenges a variety of nonlinear system identification methods have been developed, such as NARMAX methods \cite{billings2013nonlinear}, Volterra series \cite{glentis1999efficient}, Lyapunov methods \cite{parikh2018integral}, and neural networks \cite{nelles2013nonlinear}. Recent developments in nonlinear system identification include reasoning over infinite-dimensional function spaces via kernel methods \cite{SCC.Pillonetto.Dinuzzo.ea2014,SCC.Carli.Chen.ea2016,SCC.Birpoutsoukis.Marconato.ea2017,SCC.Schoukens.Ljung2019} and dynamic mode decompositions (DMDs) and their connection with the Koopman operator \cite{brunton2016discovering,hemati2014dynamic,taylor2018dynamic,SCC.Mauroy.Goncalves2020}. However, given the rich variety of nonlinear systems, there is no modal approach to resolving the system identification problem for nonlinear systems \cite{nelles2013nonlinear}.

One technical challenge that arises in many of the system identification methods described above is the estimation of the state derivative \cite{brunton2016discovering,parikh2018integral}. Frequently only the output trajectory is available and numerical estimation methods are employed to obtain samples of the state and the state derivative. Unfortunately, numerical state derivative estimates are prone to error, and introduce an artificial noise component that requires additional filtering \cite{brunton2016discovering}.

In an online parameter estimation context, \cite{parikh2018integral} leveraged the technique of integral concurrent learning, where state derivative estimates were replaced with integrals of the state. Therein it was demonstrated that the parameters were more precisely estimated via the integral concurrent learning method than by methods using state derivative estimates. Moreover, in the online setting the parameter estimation error was more robust to noise under the integral concurrent learning method \cite{parikh2018integral}.

The present manuscript develops a method to compute projection of the dynamics of a system onto the span of a collection of basis functions using samples of the system trajectory. The projection is derived for a class of dynamical systems that arise as symbols for densely defined Liouville operators over a reproducing kernel Hilbert space (RKHS). The developed method also provides a collection of constraints that can be leveraged for sparse identification routines, such as the SINDy algorithm \cite{brunton2016discovering}.

Specifically, the method presented in Section \ref{sec:systemid} leverages novel kernel techniques presented in Section \ref{sec:occupation}, where the concept of occupation kernels is introduced along side that of densely defined Liouville operators. Occupation kernels are a generalization of occupation measures, which have been used in dynamical systems theory and optimal control based largely on the seminal work of \cite{lasserre2008nonlinear}. The present manuscript links the theory of occupation measures to function theory by examining integration functionals over RKHSs rather than the Banach spaces of continuous functions.

While an occupation measure is a member of the dual of a Banach space, an occupation kernel is a function that resides in the RKHS. Moreover, the representation of a trajectory as an occupation kernel over a RKHS changes with the selection of RKHS, which allows for different aspects of the trajectory to be emphasized. In particular, an inner product on the space of dynamical systems may be determined given observed system trajectories through a combination of densely defined operators over a RKHS and the occupation kernel corresponding to the trajectory. In contrast, due to the limited availability of computational tools for measures, the study of occupation measures has been limited to polynomials in both the dynamics of the dynamical systems as well as the test functions leveraged to provide constraints on the occupation measures themselves.

The contributions of this manuscript are presented below.
\begin{itemize}
    \item The concept of Liouville operators is integrated with the theory of RKHSs to yield a representation of nonlinear dynamical systems in a Hilbert space setting.
    \item Occupation measures are generalized to occupation kernels, where a trajectory is represented inside a Hilbert space as a function.
    \item An inner product on dynamical systems that give rise to densely defined Liouville operators over RKHSs is developed in Section \ref{sec:testfree} via adjoints applied to occupation kernels. 
    \item Through a factorization of the Gram matrix for projection within the inner product space, a collection of constraints is developed for a system identification method, which is presented in Section \ref{sec:systemid}. These constraints use more general test functions than polynomials, which is an advantage that arises in the use of occupation kernels over occupation measures. In particular, the test functions under consideration correspond to feature maps for the underlying kernel function.
\end{itemize}

The manuscript is organized as follows. Preliminaries necessary for the development of occupation kernels and densely defined Liouville operators are presented in Section \ref{sec:preliminaries}, and the densely defined Liouville operators and occupation kernels themselves are introduced in Section \ref{sec:occupation}. These tools are then turned towards the development of an inner product on nonlinear dynamical systems, where the parameters resolving a system identification routine arise as the coefficients for the projection of a dynamical system onto its basis functions. Specifically, through a factorization of the Gram matrix,  a collection of linear constraints on the parameters of the dynamics in Section \ref{sec:systemid} arise naturally. These constraints may be compared with \cite{brunton2016discovering}, where state derivatives are replaced via a collection of integral constraints. Section \ref{sec:occupation-estimation} examines the convergence properties of occupation kernels associated with various numerical methods, while Section \ref{sec:noise} demonstrates a robustness to noise of the samples used in Section \ref{sec:systemid}. Section \ref{sec:streaming} examines a method for incorporating streaming data for system identification. Finally, the system identification approach is then examined through a collection of numerical experiments in Section \ref{sec:numerical} and the experiments are discussed in Section \ref{sec:discussion}.

{\color{black}\section{Problem formulation}\label{sec:systemid}

In a gray box system identification setting, the system dynamics, $f: \mathbb{R}^n \to \mathbb{R}^n$, is parameterized in terms of a collection of basis functions, $Y_i :\mathbb{R}^n \to \mathbb{R}^n$ for $i = 1, \ldots, M$, as
\begin{equation}\label{eq:parameterization}
\dot{x} = f(x) = \sum_{i=1}^M \theta_i Y_i(x).
\end{equation}
Let $X \subset \mathbb{R}^n$ be compact, $H$ be a RKHS of continuous functions over $X$, and $0 < T \in \mathbb{R}$. The goal of system identification is to determine values of the parameters, $\theta_i$ for $i=1,\ldots,M$, such that \eqref{eq:parameterization} may be used to reproduce a given a collection of continuous trajectories, $\{ \gamma_j:[0,T]\to X \}_{j=1}^{N}$.

The identification method needs to be data-efficient and robust to process and sensor noise. A noise-robust method to estimate the parameters can be derived by observing that if the trajectory $ \gamma_j $ is a solution to \eqref{eq:parameterization}, then $ \gamma_j(T) - \gamma_j(0) = \int_0^T \sum_{i=1}^M \theta_i Y_i(\gamma_j(t))dt $. A solution of the system identification problem may then be obtained by solving the integral least squares (ILS) problem
\begin{equation}
    \min_{\theta_1,\ldots,\theta_M} \sum_{j=1}^N \left\Vert\gamma_j(T) - \gamma_j(0) - \sum_{i=1}^M \theta_i \int_0^T Y_i(\gamma_j(t))dt\right\Vert_2^2.\label{eq:ILS}
\end{equation}

In this paper, the ILS problem in \eqref{eq:ILS} is generalized by posing the system identification problem as a projection problem in a RKHS. In particular, the ILS formulation is recovered when the underlying RKHS is the one corresponding to the linear kernel $ K(x,y) = x^T y $. The generalization results in several practical and theoretical advantages. By using different kernels, more information can be extracted from the same set of trajectories than the ILS formulation. For example, if only a single trajectory of a dynamical system with $ n > 1 $ is available, the regression matrix corresponding to the ILS problem is rank deficient. However, as shown in Section \ref{sec:numerical}, Experiment \ref{exp:three}, the occupation kernel method results in a well-posed least squares problem. {\color{black}As evidenced by the results of Experiment \ref{exp:five} (see Figure \ref{fig:MonteCarloILSComparison}), the method developed in this paper can yield better results than the solution of the ILS problem, even in the case where there are enough trajectories available to yield a well-conditioned regression matrix for the ILS problem.}

The generalization is realized by developing the theory of occupation kernels and Liouville operators and using it to formulating a novel inner product between dynamics that are symbols of densely defined Liouville operators. In addition to the system identification algorithm, the inner product provides a new measure of distance between two dynamical systems. In the following, fundamental properties of RKHSs are summarized before introducing the idea of occupation kernels.}

\section{Preliminaries}
\label{sec:preliminaries}
\subsection{Reproducing Kernel Hilbert Spaces}
{\color{black}
\begin{definition}A RKHS, $H$, over a set $X$ is a Hilbert space of real valued functions over the set $X$ such that for all $x \in X$ the evaluation functional, $E_x: H \to \mathbb{R}$, given as $E_x g := g(x)$ is bounded.\end{definition}}

The Riesz representation theorem guarantees, for all $x \in X$, the existence of a function $k_x \in H$ such that $\langle g, k_x \rangle_H = g(x)$, where $\langle \cdot, \cdot \rangle_H$ is the inner product for $H$ \cite[Chapter 1]{paulsen2016introduction}. The function $k_x$ is called the reproducing kernel function at $x$, and the function $K(x,y) = \langle k_y, k_x \rangle_H$ is called the kernel function corresponding to $H$.
{\color{black}
Each kernel function has an associated feature map, $\Psi : X \to \ell^2(\mathbb{N})$, such that $K(x,y) = \langle \Psi(x),\Psi(y) \rangle_{\ell^2(\mathbb{N})}$. The feature map can be obtained by using an orthonormal basis for $H$, and if $K: X \times X \to \mathbb{R}$ can be represented through a feature map, then there is a unique RKHS for which $K$ is its kernel function \cite{bishop2006pattern}.
}
This manuscript utilizes two RKHSs, which are defined through their kernel functions. For $\mu > 0$, the kernel function $K_E(x,y) = e^{\mu x^Ty}$ is called the exponential dot product kernel function, and for $\mu > 0$, the kernel function given as $K_G(x,y) = \exp\left( -\frac{1}{\mu} \|x-y\|_2^2\right)$ is called a Gaussian radial basis function. Both $K_E$ and $K_G$ are kernel functions for RKHSs over $\mathbb{R}^n$ \cite[Chapter 4]{steinwart2008support}.

\subsection{Densely Defined Operators}

For many RKHSs of continuously differentiable functions, the differential operator is unbounded, which means that there are frequently functions in such a RKHS, $H$, such that their derivative is not a member of $H$. The focus of this manuscript is in the study of Liouville operators, which implement the gradient operation on members of a RKHS. As such, care will be required in defining these operators and their domain.

Given a Hilbert space, $H$, and a subspace $\mathcal{D}(T) \subset H$ a linear operator $T:\mathcal{D}(T) \to H$ is called densely defined if $\mathcal{D}(T)$ is a dense subspace of $H$ \cite[Chapter 5]{pedersen2012analysis}. The operator $T$ is closed if, for every sequence $\{ g_m \}_{m=0}^\infty \subset \mathcal{D}(T)$, such that $g_m \to g \in H$ and $T g_m \to h \in H$, then $g \in \mathcal{D}(T)$ and $Tg = h$.

The adjoint of a possibly unbounded operator is given first by its domain: $\mathcal{D}(T^*)$$ = \{ g \in H : h \mapsto \langle Th, g \rangle_H \text{ is bounded over } \mathcal{D}(T) \}$ \cite[Chapter 5]{pedersen2012analysis}. For each $g \in \mathcal{D}(T^*)$ there exists a member $T^*g \in H$ such that $\langle Th,g \rangle_H = \langle h, T^*g\rangle_H$. Thus, the operator $T^*$ may be defined as taking $g \in \mathcal{D}(T^*)$ to $T^*g$, which was obtained through the Riesz representation theorem. The closedness of the operator guarantees the nonemptiness of the domain of its adjoint. In fact, the following stronger statement holds.

\begin{lemma}(c.f. \cite[Chapter 5]{pedersen2012analysis})\label{lem:closedops} The adjoint of a closed operator is densely defined.\end{lemma}

\section{Liouville Operators and Occupation Kernels}\label{sec:occupation}

To establish a connection between RKHSs and nonlinear dynamical systems, the following operator is introduced, which is inspired by the study of occupation measures \cite{lasserre2008nonlinear}.

\begin{definition}Let $\dot x = f(x)$ be a dynamical system with the dynamics, $f: \mathbb{R}^n \to \mathbb{R}^n$, locally Lipschitz continuous, and suppose that $H$ is a RKHS over a set $X$, where $X \subset \mathbb{R}^n$ is compact. The \emph{Liouville operator with symbol $f$}, $A_f : \mathcal{D}(A_f) \to H$, is given as \[A_f g := \nabla_x g \cdot f,\] where \[\mathcal{D}(A_f):= \left\{ g \in H : \nabla_x g \cdot f \in H \right\}. \] \end{definition}

A small subset of Liouville operators can be obtained as infinitesimal generators for semigroups of Koopman operators. For a Liouville operator to be such an infinitesimal generator, the associated symbol (or continuous time dynamics) must be forward invariant (cf. \cite[Assumption 1]{giannakis2018reproducing}). Specifically, dynamics that lead to finite escape times, such as $\dot x = 1+x^2$, cannot be discretized to yield a Koopman operator, and as a result, the Liouville operator corresponding to the symbol $f(x) := 1+x^2$ cannot be obtained as an infinitesimal generator. More generally, Liouville operators corresponding to globally Lipschitz symbols can be obtained as infinitesimal generators. However, techniques that rely on global Lipschitz continuity would exclude most polynomial systems and other nonlinear systems with finite escape times. Hence, methods that allow generators to inherit desirable properties, such as being densely defined and closed, from their associated semigroups, cannot be leveraged to study a majority of Liouville operators. These properties must be established separately to ensure full generality.

Liouville operators embed the nonlinear dynamics inside of an unbounded operator. The first question to address is that of existence. In particular, are there reasonable classes of dynamics for which the Liouville operator is densely defined over a RKHS? 

\begin{example}
The most commonly investigated dynamical systems are those with polynomial dynamics. In the case that $f$ is a polynomial over $\mathbb{R}^n$, a Liouville operator with those dynamics maps polynomials to polynomials, when polynomials are contained in the RKHS in question. One example, where polynomials are not only contained in the RKHS but are also dense is the exponential dot product kernel's native RKHS \cite[Chapter 4]{steinwart2008support}. Moreover, for this space, the collection of monomials forms an orthogonal basis.
\end{example}

The above example guarantees the existence of densely defined Liouville operators for a large class of dynamics. Adjusting the RKHS will also adjust the Liouville operators that are admissible. In the case when a Liouville operator is not known to be densely defined, some of the methods of this manuscript may still be applied as a heuristic algorithm.

As a differential operator, $A_f$ is not expected to be a bounded over any RKHS. However, as differentiation is a closed operator over RKHSs consisting of continuously differentiable functions, which follows from \cite[Corollary 4.36]{steinwart2008support}, it can be similarly established that $A_f$ is closed under the same circumstances.

\begin{theorem}Let $H$ be a RKHS of continuously differentiable functions over a set $X$ and $f:\mathbb{R}^n\to\mathbb{R}^n$ be a function such that $A_f$ has nontrivial domain, then $A_f$ is a closed operator.\end{theorem}

\begin{proof}
By \cite[Corollary 4.36]{steinwart2008support}, it can be observed that if $\{ g_m\}_{m=1}^\infty \subset H$ such that $\| g_m - g \|_H \to 0$ in $H$ then $\left\{ \frac{\partial}{\partial x_i} g_m \right\}_{m=0}^\infty$ converges to $\frac{\partial}{\partial x_i} g$ uniformly in $X$. Hence, if $\{ g_m \}_{m=0}^\infty \subset \mathcal{D}(A_f) \subset H$ converges to $g$ and $\{ A_f g_m \}_{m=0}^\infty$ converges to $h \in H$ then $\nabla_x g_m(x)f(x)$ converges to $\nabla_x g(x)f(x)$ pointwise. As $A_f g_m(x) = \nabla_x g_m(x)f(x)$, it follows that $h(x) = \lim_{m\to\infty} A_f g_m(x) = \nabla_x g(x)f(x)$. By the definition of $\mathcal{D}(A_f)$, $g \in \mathcal{D}(A_f)$ and $A_f g = h$.
\end{proof}

Thus, $A_f$ is a closed operator for RKHSs consisting of continuously differentiable functions. Consequently, the adjoints of densely defined Liouville operators are themselves densely defined by Lemma \ref{lem:closedops}. Associated with Liouville operators in particular are a special class of functions within the domain of the Liouville operators' adjoints, and these functions are also the main object of study of this manuscript.

\begin{definition}\label{def:occ}Let $X \subset \mathbb{R}^n$ be compact, $H$ be a RKHS of continuous functions over $X$, and $\gamma:[0,T] \to X$ be a continuous trajectory. The functional $g \mapsto \int_0^T g(\gamma(\tau)) d\tau$ is bounded {\color{black}over H}, and may be respresented as $\int_0^T g(\gamma(\tau)) d\tau = \langle g, \Gamma_{\gamma}\rangle_H,$ for some $\Gamma_{\gamma} \in H$ by the Riesz representation theorem. The function $\Gamma_{\gamma}$ is called the occupation kernel corresponding to $\gamma$ in $H$.\end{definition}

\begin{proposition}\label{prop:occupation}
Let $H$ be a RKHS of continuously differentiable functions over a compact set $X$, and suppose that $f:\mathbb{R}^n \to \mathbb{R}^n$ is Lipschitz continuous. If $\gamma:[0,T] \to X$ is a trajectory as in Definition \ref{def:occ} that satisfies $\dot \gamma = f(\gamma)$, then $\Gamma_{\gamma} \in \mathcal{D}(A_f^*)$, and $A_f^* \Gamma_{\gamma} = K(\cdot,\gamma(T)) - K(\cdot,\gamma(0))$.
\end{proposition}

\begin{proof}
The establishment of $\Gamma_{\gamma} \in \mathcal{D}(A_f^*)$, requires the demonstration that the functional $g \mapsto \langle A_f g, \Gamma_{\gamma} \rangle_H$ is bounded over $\mathcal{D}(A_f)$. Note that
\begin{equation}\label{eq:total-derivative}
    \int_{0}^T \nabla_x g(\gamma(t)) f(\gamma(t)) dt = g(\gamma(T))-g(\gamma(0)) = \langle g, K(\cdot, \gamma(T))-K(\cdot, \gamma(0))\rangle_H
\end{equation}
as the integrand of \eqref{eq:total-derivative} is the total derivative of $g(\gamma(t))$. The left hand side of \eqref{eq:total-derivative} may be expressed as $\langle A_f g, \Gamma_{\gamma} \rangle_H$, while the right hand side satisfies the bound 
\begin{gather*}|g(\gamma(T))-g(\gamma(0))| = |\langle g, K(\cdot,\gamma(T))-K(\cdot,\gamma(0))\rangle_H|\\
\le \|g\|_H \| K(\cdot,\gamma(T))-K(\cdot,\gamma(0)) \|_H,
\end{gather*}
which establishes the boundedness of $g \mapsto \langle A_f g, \Gamma_{\gamma} \rangle_H$. Moreover, since \[ \langle A_f g, \Gamma_{\gamma} \rangle_H = \langle g, K(\cdot, \gamma(T))-K(\cdot, \gamma(0))\rangle_H, \] it follows that $A_f^* \Gamma_{\gamma} =  K(\cdot, \gamma(T))-K(\cdot, \gamma(0))$
\end{proof}

Proposition \ref{prop:occupation} completes the integration of nonlinear dynamical systems with RKHSs. In particular, valid trajectories for the dynamical system appear as occupation kernels within the domain of the adjoint of the Liouville operator corresponding to the dynamics. This intertwining allows for the expression of finite dimensional nonlinear dynamics as linear systems in infinite dimensions.

Moreover, the relation \[\langle A_f g, \Gamma_{\gamma} \rangle_H = g(\gamma(T)) - g(\gamma(0)) \text{ for all } g \in \mathcal{D}(A_f) \] uniquely determines $\Gamma_{\gamma}$. Consequently, this relation will be used subsequently to establish constraints for parameter identification in a system identification setting.

While occupation kernels representing solutions to the dynamical systems have a clear image under the adjoint of the corresponding Liouville operators, it can be demonstrated that all occupation kernels corresponding to continuous trajectories are also in the domain of the adjoint of Liouville operators.
\begin{theorem}\label{thm:adjoint-on-occupation-kernel}
Let $H$ be a RKHS of continuously differentiable functions over a set $X \subset \mathbb{R}^n$, let $f$ be a symbol for a densely defined Liouville operator, $A_{f} : \mathcal{D}(A_f) \to H$, and let $\gamma:[0,1] \to X$ be a continuously differentiable trajectory. Then $\Gamma_{\gamma} \in \mathcal{D}(A_f^*)$ and for each $x_0 \in X$, $K(\cdot,x_0) \in \mathcal{D}(A_f^*)$. Moreover, the images of each function may be expressed as
\begin{gather*}
    A_{f}^* K(\cdot,x_0) = \nabla_2 K(\cdot,x_0) f(x_0), \text{ and}\\
    A_{f}^* \Gamma_\gamma = \int_{0}^T \nabla_2 K(\cdot,\gamma(t)) f(\gamma(t))dt,
\end{gather*}
where $\nabla_2$ indicates that the gradient is performed on the second variable.
\end{theorem}

\begin{proof}
For each $i=1,\ldots,n$, the boundedness of the linear functional $g \mapsto \left\langle \frac{\partial}{\partial x_i} g, K(\cdot,x_0)\right\rangle_H$ follows from \cite[Corollary 4.36]{steinwart2008support}. Let $\{ e_m \}_{m=1}^\infty \subset H$ be an orthonormal basis for $H$, then $\frac{\partial}{\partial x_i} K(\cdot, \cdot) = \sum_{m=1}^\infty e_m(\cdot) \frac{\partial}{\partial x_i} e_m(\cdot)$. More generally, if $g = \sum_{m=1}^\infty a_m e_m$ then $\frac{\partial}{\partial x_i} g(x_0) = \sum_{m=1}^\infty a_m \frac{\partial}{\partial x_i} e_m(x_0)$. Thus,
\[\frac{\partial}{\partial x_i} g(x_0) f_i(x_0) = \left\langle g, \frac{\partial}{\partial x_i} K(\cdot,x_0) \right\rangle_H f_i(x_0) = \left\langle g, \frac{\partial}{\partial x_i} K(\cdot,x_0) f_i(x_0) \right\rangle_H. \]
It follows that $A_f^* K(\cdot,x_0) = \nabla_2 K(\cdot, x_0) f(x_0)$ via a linear combination. The formula for $A_{f}^* \Gamma_\gamma$ follows from a limiting argument leveraging Proposition \ref{prop:quadrature-convergence} below and the closedness of $A_{f}^*$.
\end{proof}

{\color{black}The last theorem of this section addresses a characterization of symbols that give rise to densely defined Liouville operators over RKHSs consisting of real analytic functions. Many frequently used RKHSs are precisely of this form, including the exponential dot product kernels and the Gaussian RBF kernel's native spaces \cite{steinwart2008support,lu2005positive}. This characterization is essential for the generation of an inner product in Section \ref{sec:testfree}.

\begin{theorem}\label{thm:real-analytic-symbol}
If $f$ is the symbol for a densely defined Liouville operator over a RKHS consisting of real analytic functions of several variables over a set $X \subset \mathbb{R}^n$ such that the collection of gradient of the functions in the space are universal in $\mathbb{R}^n$, then $f$ must be a vector valued real analytic function of several variables.
\end{theorem}

\begin{remark}
The universality of the gradients of the kernels are readily established for the exponential dot product kernel, where the gradients of the monomials will span the collection of vectors of monomials in $\mathbb{R}^n$. The Gaussian RBF's universality can be found in \cite{caponnetto2008universal}.
\end{remark}

\begin{proof}

Write $f = (f_1,\ldots,f_n)^T$.

Given any $v \in \mathbb{R}^n{\setminus \{0\}}$ and $x \in X$, the proof of Theorem \ref{thm:adjoint-on-occupation-kernel} established that the functional $g \mapsto \nabla g(x) v$ is bounded as real analytic functions are continuously differentiable. Hence, there is a function $h_{x,v}$ that represents that functional through the inner product of the RKHS. The universality of the gradients of functions in the Hilbert space guarantees that there is at least one function for which $\nabla g(x) v$ is nonzero. Hence, $h_{x,v}$ is not the zero function for any {$v \in \mathbb{R}^n\setminus \{0\}$ and $x \in \mathbb{R}^n$}.

Select $x \in X$, then $W_x := \vspan \{ \nabla g(x) \}_{g \in \mathcal{D}(A_f)} = \mathbb{R}^n$. If not, then there is a {$v \in \mathbb{R}^n\setminus\{0\}$} such that $\nabla g(x) v = 0$ for all $g \in \mathcal{D}(A_f)$, and hence, $\langle g, h_{x,v} \rangle_{H} = 0$ for all $g \in \mathcal{D}(A_f)$. As a result, $\mathcal{D}(A_f)$ has codimension at least 1 in $H$ and is not dense (which is a contradiction).

Thus, for a fixed $x_0 \in \mathbb{R}^n$, a complete basis for $\mathbb{R}^n$ may be selected from $W_{x_0}$ as $\nabla g_1({x_0}), \ldots, \nabla g_n({x_0})$, and there are linear combinations of these vectors that yield the standard basis in $\mathbb{R}^n$ at the point $x_0$. In particular, this means that   $x\mapsto\det(B(x))$, with 
$B(x) := \begin{pmatrix}
        \nabla g_1(x)^T&
        \cdots &
        \nabla g_n(x)^T
    \end{pmatrix}^T,$
is an analytic function that is nonvanishing at $x_0$ (by linear independence).
The analyticity follows since products and sums of real analytic functions are real analytic, and each component of $\nabla g_i(x)$ is a real analytic function. 

Now consider $G_i(x) := A_f g_i(x) = \nabla g_i(x) f(x)$. Let $B_i(x)$ be matrix obtained by replacing the $i$-th column of $B(x)$ by the column vector consisting of the functions $G_i(x)$. Then, $\det(B_i(x))$ is also a real analytic function.

Finally, by Cramer's rule, $f_i(x) = \det(B_i(x))/\det(B(x))$, and since $\det(B(x_0)) \neq 0$, $f_i$ is real analytic at $x_0$. As $x_0$ was arbitrary, $f_i$ is real analytic everywhere.
\end{proof}
}

\section{Estimation of Occupation Kernels}\label{sec:occupation-estimation}

Approximating the value of an inner product against an occupation kernel in a RKHS can be performed leveraging quadrature techniques for integration. The occupation kernels themselves can be expressed as an integral against the kernel function in a RKHS as demonstrated in Proposition \ref{prop:integral-rep}. 

\begin{proposition}\label{prop:integral-rep}
Let $H$ be a RKHS over a compact set $X$ consisting of continuous functions and let $\gamma : [0,T] \to X$ be a continuous trajectory as in Definition \ref{def:occ}. The occupation kernel corresponding to $\gamma$ in $H$, $\Gamma_{\gamma}$, may be expressed as 
\begin{equation}\label{eq:integral-rep}\Gamma_\gamma(x) = \int_0^T K(x,\gamma(t)) dt.\end{equation}
\end{proposition}

\begin{proof}
Note that $\Gamma_\gamma(x) = \langle \Gamma_\gamma, K(\cdot,x)\rangle_H$, by the reproducing property of $K$. Consequently,
\begin{gather*}
    \Gamma_\gamma(x) = \langle \Gamma_\gamma, K(\cdot,x)\rangle_H= \langle K(\cdot,x), \Gamma_\gamma \rangle_H\\
     = \int_0^T K(\gamma(t),x) dt
    = \int_0^T K(x,\gamma(t)) dt,
\end{gather*}
which establishes the result.
\end{proof}

Leveraging Proposition \ref{prop:integral-rep}, quadrature techniques can be demonstrated to give not only pointwise convergence but also norm convergence in the RKHS, which is a strictly stronger result.

\begin{proposition}\label{prop:quadrature-convergence}
Under the hypothesis of Proposition \ref{prop:integral-rep}, let $t_0 = 0 < t_1 < t_2 < \ldots < t_{F} = T$, suppose that $\gamma$ is a continuously differentiable trajectory and $H$ is composed of continuously differentiable functions. Consider 
\begin{equation}\label{eq:quadrature-rep}
\hat \Gamma_\gamma(x) := \sum_{i=1}^F (t_{i} - t_{i-1}) K(x,\gamma(t_i)).
\end{equation}
The norm distance is bounded as $\| \Gamma_{\gamma} - \hat \Gamma_{\gamma} \|_H^2 = O(h),$ where $h = \max_{i=1,\ldots,F}|t_{i} - t_{i-1}|$.
\end{proposition}

\begin{proof}
Consider, \[\| \Gamma_{\gamma} - \hat \Gamma_{\gamma} \|_H^2 = \|\Gamma_{\gamma}\|^2 + \| \hat \Gamma_{\gamma} \|^2 - 2 \langle \Gamma_{\gamma}, \hat \Gamma_{\gamma} \rangle_H.\] The norm of the approximation can be expanded as
\begin{gather}\|\hat \Gamma_\gamma \|_H^2 = \langle \hat \Gamma_{\gamma},\hat \Gamma_{\gamma}\rangle_H = \nonumber\\ \label{eq:quadrature-approx}
\sum_{i=1}^F \sum_{j=1}^F (t_{i} - t_{i-1}) (t_{j} - t_{j-1}) K(\gamma(t_j), \gamma(t_i))\end{gather}
via the reproducing property of $K$. Now compare each term in \eqref{eq:quadrature-approx} to the corresponding integral, 
\begin{gather}\int_{t_{i-1}}^{t_{i}} \int_{t_{j-1}}^{t_j} K(\gamma(t),\gamma(\tau)) dt d\tau \nonumber\\\label{eq:integral-difference} - (t_{i} - t_{i-1}) (t_{j} - t_{j-1}) K(\gamma(t_j), \gamma(t_i)).
\end{gather}
By the mean value theorem, there is a point $(\tau^*,t^*) \in [t_{i-1},t_i] \times [t_{j-1},t_j]$ such that 
\begin{gather*}\int_{t_{i-1}}^{t_{i}} \int_{t_{j-1}}^{t_j} K(\gamma(t),\gamma(\tau)) dt d\tau\\ = (t_{i} - t_{i-1}) (t_{j} - t_{j-1}) K(\gamma(t^*), \gamma(\tau^*)).\end{gather*}
Hence, \eqref{eq:integral-difference} may be written as 
\begin{gather*}
    (t_{i} - t_{i-1}) (t_{j} - t_{j-1})(K(\gamma(t^*), \gamma(\tau^*)) - K(\gamma(t_j), \gamma(t_i))).
\end{gather*}
Leveraging the mean value inequality \cite{rudin1964principles}, 
\begin{gather*}
    |K(\gamma(t^*), \gamma(\tau^*)) - K(\gamma(t_j), \gamma(t_i))| \le\\ \sup_{x,y \in X} \| \nabla K(x,y) \|_2 \max_{0 < t < T} | \dot \gamma(t) | \| (\tau^*, t^*) - (t_i,t_j) \|_2.
\end{gather*}
Taking $h = \max_{i=1,\ldots,F} | t_{i} - t_{i-1}|$ and combining the above equations, it can be observed that \[\| \hat \Gamma_{\gamma}\|_H^2 = \| \Gamma_{\gamma}\|_H^2 + O(h).\]

Similarly, it may be demonstrated that $\langle \hat \Gamma_\gamma, \Gamma_{\gamma} \rangle_H = \| \Gamma_{\gamma}\|_{H}^2 + O(h),$ and the conclusion follows. 
\end{proof}

It should be clear from the proof of Proposition \ref{prop:quadrature-convergence} that higher order quadrature rules for estimating the integral in \eqref{eq:integral-rep} will also lead to higher order convergence rates of the difference in Hilbert space norms of the occupation kernel and the quadrature estimate with the added caveat of higher order continuous differentiability of the kernels and trajectories. For example, Simpson's Rule is a quadrature method that yields a convergence rate of $O(h^4)$ \cite{atkinson2008introduction}, and the following theorem captures obtained convergence rate for the corresponding approximation of the occupation kernel.

\begin{theorem}\label{thm:simpsons-convergence}
Under the hypothesis of Proposition \ref{prop:integral-rep}, let $t_0 = 0 < t_1 < t_2 < \ldots < t_{F} = T$ (with $F$ even and $t_i$ evenly spaced), suppose that $\gamma$ is a fourth order continuously differentiable trajectory and $H$ is composed of fourth order continuously differentiable functions. Set $h$ to satisfy $t_i = t_0 + ih$, and consider 
\begin{gather}\label{eq:quadrature-rep-simpsons}
\hat \Gamma_\gamma(x) := \frac{h}{3}\left(K(x,\gamma(t_0)) + 4 \sum_{i=1}^{\frac{F}{2}} K(x,\gamma(t_{2\cdot i - 1}))\right. \nonumber\\
\left.+ 2 \sum_{i=1}^{\frac{F}{2}-1} K(x,\gamma(t_{2\cdot i})) + K(x,\gamma(t_F)) \right).
\end{gather}
The norm distance is bounded as $\| \Gamma_{\gamma} - \hat \Gamma_{\gamma} \|_H^2 = O({\color{black}h^4}).$
\end{theorem}

\begin{proof}
Consider $\| \Gamma_{\gamma} - \hat \Gamma_{\gamma} \|_H^2 = \langle \Gamma_\gamma, \Gamma_\gamma \rangle_H + \langle \hat\Gamma_\gamma, \hat\Gamma_\gamma \rangle_H - 2 \langle \Gamma_\gamma, \hat\Gamma_\gamma \rangle_H$. The term $\langle \hat\Gamma_\gamma, \hat\Gamma_\gamma \rangle_H$ is an implementation of the two-dimensional Simpson's rule (cf. \cite{burden2001numerical}) while $\langle \Gamma_\gamma, \Gamma_\gamma \rangle_H$ is the double integral $\int_0^T \int_0^T K(\gamma(t),\gamma(\tau)) dt d\tau.$ Thus, \[\langle \hat\Gamma_\gamma, \hat\Gamma_\gamma \rangle_H = \langle \Gamma_\gamma, \Gamma_\gamma \rangle_H + O({\color{black}h^4}).\]
Similarly, $\langle \Gamma_{\gamma},\hat \Gamma_{\gamma}\rangle_H$ integrates in one variable while implementing Simpson's rule in the other. Consequently, \[\langle \Gamma_{\gamma},\hat \Gamma_{\gamma}\rangle_H = \langle \Gamma_{\gamma}, \Gamma_{\gamma}\rangle_H + O({\color{black}h^4}).\]
The conclusion of the theorem follows. 
\end{proof}

As convergence properties of occupation kernels in connection with convergence properties of the trajectories they represent are of interest in this manuscript, additional propositions have been included in the appendix which address homotopic parameterizations of curves and their respective occupation kernels.

\section{Inner Products on Symbols of Densely Defined Liouville Operators}\label{sec:testfree}

{\color{black}This section presents a method for parameter identification that builds on the Hilbert space and operator theoretic framework presented in Section \ref{sec:occupation}. In particular, the development of this section uses the adjoint relation between Liouville operators and occupation kernels to establish an inner product on the collection of symbols for densely defined Liouville operators, $\mathcal{F}$, over a RKHS. 
This section begins with a pre-inner product arising from a single trajectory for the dynamical system, and then develops two different inner products based on this pre-inner product.

Let $\gamma:[0,T] \to \mathbb{R}^n$ satisfying $\dot \gamma = f(\gamma)$. As the identity $A_f^* \Gamma_{\gamma} = \sum_{m=1}^M \theta_m A_{Y_m}^* \Gamma_\gamma$ holds the following quadratic form arises as
\begin{gather}\label{eq:occupation-direct}
0 = \left\| A_f^* \Gamma_{\gamma} - \sum_{m=1}^M \theta_m A_{Y_m}^* \Gamma_{\gamma}\right\|^2_H \\
\nonumber
= \left\| A_f^* \Gamma_{\gamma}\right\|^2_H - 2 \sum_{m=1}^M \theta_{m} \langle A_f^* \Gamma_{\gamma}, A_{Y_m}^* \Gamma_{\gamma} \rangle_H + \sum_{m,m'=1}^M \theta_{m}\theta_{m'} \langle A_{Y_m}^* \Gamma_{\gamma}, A_{Y_{m'}}^* \Gamma_{\gamma} \rangle_H.
\end{gather}
The challenge in leveraging \eqref{eq:occupation-direct} to generate constraints on $\theta$ for system identification lies in the ability to compute the various elements of \eqref{eq:occupation-direct}. For example the first two terms can be computed without appealing to a particular kernel space or determining the adjoint operator $A_{Y_m}^*$ as 
\begin{gather*}\| A_f^* \Gamma_{\gamma} \|^2_H = \langle K(\cdot, \gamma(T)) - K(\cdot, \gamma(0)), K(\cdot, \gamma(T)) - K(\cdot, \gamma(0)) \rangle_H\\
= K(\gamma(T),\gamma(T)) - 2 K(\gamma(T),\gamma(0)) + K(\gamma(0),\gamma(0)), \text{ and }\\
\langle A_{f}^* \Gamma_\gamma, A_{Y_m}^* \Gamma_{\gamma} \rangle_H = \int_0^T A_{Y_m} ( K(\cdot,\gamma(T)) - K(\cdot, \gamma(0)))\mid_{\gamma(t)} dt,
\end{gather*}
where the notation $(\cdot)\mid_{\gamma(t)}$ stands for evaluation at $\gamma(t)$. 

Closed form expressions of $A_{Y_m}^*$ are not expected to be easily determined for most choices of $Y_m$ and RKHS (e.g. \cite{russorosenfeldarxiv}). Hence, computation of the third term in \eqref{eq:occupation-direct} relies on the action of $A_{Y_m}^*$ on an occupation kernel (which is easier to determine, and was given in Theorem \ref{thm:adjoint-on-occupation-kernel}) as
\begin{equation}\label{eq:innerproduct_testfree}
    \langle A_{Y_m}^* \Gamma_{\gamma}, A_{Y_{m'}}^* \Gamma_{\gamma} \rangle_H =
    \int_0^T \int_{0}^T \nabla_1 \left( \nabla_{2} K(\gamma(t) , \gamma(\tau))Y_{m}(\gamma(\tau)) \right) Y_{m'}(\gamma(t)) d\tau dt.
\end{equation}

Hence, parameter identification can be performed using only the occupation kernel and the Liouville operators by setting the gradient of \eqref{eq:occupation-direct} equal to zero. As the norm squared in \eqref{eq:occupation-direct} is zero at the true parameters of the system, which is the smallest value the norm can take, this must be the minimum of the quadratic equation in \eqref{eq:occupation-direct}, and hence the true parameters must also set the gradient equal to zero. The parameters $\theta_1, \ldots, \theta_M$ must then satisfy
\begin{equation}\label{eq:testfree_gram}
    \begin{pmatrix}
        \langle Y_1, Y_{1} \rangle_{\mathcal{F},\gamma} & \cdots & \langle Y_1, Y_{M} \rangle_{\mathcal{F},\gamma}\\
        \vdots & \ddots & \vdots\\
        \langle Y_M, Y_{1} \rangle_{\mathcal{F},\gamma} & \cdots & \langle Y_M, Y_{M} \rangle_{\mathcal{F},\gamma}
    \end{pmatrix}
    \begin{pmatrix}
        \theta_1\\
        \vdots\\
        \theta_M
    \end{pmatrix}
    =
    \begin{pmatrix}
        \langle {f}, Y_{1} \rangle_{\mathcal{F},\gamma}\\
        \vdots\\
        \langle {f}, Y_{M} \rangle_{\mathcal{F},\gamma}
    \end{pmatrix},
\end{equation}
where for a collection, $\mathcal{F}$, of symbols of densely defined Liouville operators a bilinear form is given as $\langle f, g \rangle_{\mathcal{F},\gamma} := \langle A_g^* \Gamma_\gamma,A_f^* \Gamma_{\gamma}\rangle_{H}$, which gives a pre-inner product on the space of dynamical systems giving rise to densely defined Liouville operators over $H$. Note that, in contrast with the SINDy method found in \cite{brunton2016discovering}, \eqref{eq:testfree_gram} yields a derivative free approach for the system identification problem. Here, the only derivatives to be performed are those that can be computed symbolically, which are easily computed for many kernel functions. Moreover, the above formulation only leverages a single trajectory of the system, and the size of the Gram matrix corresponding to \eqref{eq:occupation-direct}, namely $G := \left( \langle Y_m, Y_{m'} \rangle_{\mathcal{F},\gamma} \right)_{m,m'=1}^M$, is governed only by the number of basis functions. The matrix, $G$, is positive semidefinite. In particular, the resolution of \eqref{eq:testfree_gram} gives the projection of an arbitrary $f \in \mathcal{F}$ onto the span of $\{ Y_{1}, \cdots, Y_{M}\}$ with respect to the pre-inner product $\langle \cdot, \cdot, \rangle_{\mathcal{F},\gamma}$. For completeness, analytical evaluation of the pre-inner products expressed in \eqref{eq:innerproduct_testfree} is presented in Appendix \ref{sec:ExplicitInnerproducts} for several kernels.

\subsection{Computational Challenges and Feature Space Representations}

The resolution of the weights in the above setting expresses the projection of $f$ onto the span of $\{ Y_{1}, \ldots, Y_{M} \}$ with respect to the pre-inner product space given by $\langle \cdot, \cdot \rangle_{\mathcal{F},\gamma}$. However, the conditioning of the Gram matrix, $( \langle Y_{i}, Y_{j} \rangle_{\mathcal{F},\gamma} )_{i,j=1}^M$, is frequently poor. One approach to robustify the result is to recognize that the weights are unchanged when the inner product is adjusted between two different trajectories satisfying the same dynamics. The respective linear systems may then be augmented, and a left psuedo-inverse may be employed to determine the parameters for the system.

However, as is typical of numerical methods and matrix computations (e.g. \cite{fasshauermccourt,williams2014kernel}), a more reliable result may be extracted via a simplification obtained through a factorization of the Gram matrices. If the kernel $K$ is associated with the feature map $\Psi(x) = (\Psi_1(x),\Psi_2(x),\ldots)^T \in \ell^2(\mathbb{N})$ as $K(x,y) = \sum_{s=1}^\infty \Psi_s(x) \Psi_s(y)$, then $A_{Y_i}^* \Gamma_\gamma$ and $A_{f}^* \Gamma_{\gamma}$ may be written as 
\begin{gather}
A_{Y_i}^* \Gamma_\gamma = \sum_{s=1}^\infty \Psi_s(x) \int_0^T \nabla \Psi_s(\gamma(t)) Y_i(\gamma(t)) dt\text{, and}  \label{eq:featurespace_Yi}\\
A_{f}^* \Gamma_{\gamma} = \sum_{s=1}^\infty \Psi_s(x) \left(\Psi_s(\gamma(T)) - \Psi_s(\gamma(0))\right). \label{eq:featurespace_f}
\end{gather}
Hence, the Gram matrix on the left hand side of \eqref{eq:testfree_gram} may be expressed as
$V_\gamma^T V_\gamma$ with
\begin{equation}\label{eq:factored_gram}
    V_\gamma := \begin{pmatrix}
        \int_0^T \nabla \Psi_1(\gamma(t)) Y_1(\gamma(t)) dt &  \cdots & \int_0^T \nabla \Psi_1(\gamma(t)) Y_M(\gamma(t)) dt\\
        \int_0^T \nabla \Psi_2(\gamma(t)) Y_1(\gamma(t)) dt & \cdots & \int_0^T \nabla \Psi_2(\gamma(t)) Y_M(\gamma(t)) dt\\
        \int_0^T \nabla \Psi_3(\gamma(t)) Y_1(\gamma(t)) dt & \cdots & \int_0^T \nabla \Psi_3(\gamma(t)) Y_M(\gamma(t)) dt\\
        \vdots & \vdots & \vdots
    \end{pmatrix}.
\end{equation}
The right hand side of \eqref{eq:testfree_gram} is expressible as
\begin{equation}\label{eq:factored_eval}
\begin{pmatrix}
\langle f, Y_1 \rangle_{\mathcal{F},\gamma}\\
\langle f, Y_2 \rangle_{\mathcal{F},\gamma}\\
\vdots\\
\langle f, Y_M \rangle_{\mathcal{F},\gamma}
\end{pmatrix}
=
V_\gamma^T (\Psi(\gamma(T)) - \Psi(\gamma(0)))
=
V_\gamma^T
\begin{pmatrix} 
\Psi_1(\gamma(T)) - \Psi_1(\gamma(0))\\
\Psi_2(\gamma(T)) - \Psi_2(\gamma(0))\\
\Psi_3(\gamma(T)) - \Psi_3(\gamma(0))\\
\vdots
\end{pmatrix}.
\end{equation}
Since both \eqref{eq:factored_eval} and \eqref{eq:factored_gram} have the infinite matrix $V_\gamma^T$ on the left hand side, the resolution of the system
\begin{equation}
    V_\gamma \theta = \Psi(\gamma(T)) - \Psi(\gamma(0))\label{eq:infiniteDimensionalConstraints}
\end{equation}
also satisfies \eqref{eq:testfree_gram}.

As $\Psi$ is infinite dimensional for most kernel functions, $V_\gamma$ is an infinite dimensional matrix. For a given kernel function, such as the Gaussian RBF kernel, one option to obtain a finite-dimensional representation is to leverage decaying factors for the feature space representation and set a cutoff after the size of the features fall under a pre-specified precision, as was done in \cite{fasshauermccourt} for scattered data interpolation. Conversely, given a finite collection of real-valued functions $\{ g_1, \ldots, g_S \}$ over a set $X$, a kernel function, $K(x,y) = \sum_{s=1}^S g_s(x) g_s(y)$, may be constructed yielding a matrix $V$ of finite dimensions. In the sequel, a collection of test functions will be employed for the resolution of the system identification problem given in \eqref{eq:parameterization}. These test functions can be any collection of continuously differentiable functions, provided that the collection is either finite or constitutes the members of a feature map to $\ell^2(\mathbb{N})$. Each collection of test functions gives rise to a kernel function and in turn, give a different inner product on the collection of densely defined Liouville operators over the native space of that kernel function.

It should also be noted that if $f$ is known to be explicitly representable as a linear combination of a finite number of $Y_i$'s, then the matrix $V_{\gamma}$ only needs to be evaluated up until its rank matches the number of basis functions, $M$. At that point, $\theta$ is completely determined.

\subsection{Inner Products from Pre-Inner Products}

Pre-inner products give rise to pseudonorms on vector spaces. It can be seen that when $\nabla_1 \nabla_2 (K(x,y) + K(y,x))$ is positive definite and bounded below for all $x,y \in \gamma [0,T]$, as it is for \eqref{eq:pre-inner-exp-dot}, \eqref{eq:pre-inner-gauss-rbf}, and \eqref{eq:pre-inner-poly}, then the pseudonorm induced by the pre-inner product can be seen to dominate the $L^2$ norm on $\gamma[0,T]$. To give rise to densely defined Liouville operators, the dynamics in the cases of the native RKHSs of the Gaussian RBF and the exponential dot product kernels must be real analytic functions of several variables. As a result, agreement in the pseudonorm implies agreement in the $L^2$ norm over a trajectory. However, since functions, say $\eta(x)$, that vanish identically on the trajectory cannot be observed through this pseudonorm, there remains some ambiguity, where $f(x)$ and $f(x) + \eta(x)$ are different functions over $\mathbb{R}^n$ but are indistinguishable according to the induced pseudonorm. For example, $\eta(x) := x_1^2 + x_2^2 - 1$ is identically zero on the unit circle in $\mathbb{R}^2$.

It should be noted that in this setting, if only the values along the trajectory were desired, then the pre-inner product is a proper inner product. However, since it is desirable to extend the data outside of the trajectory to all of $\mathbb{R}^n$, adjustments are necessary to achieve an inner product of this form.

\subsubsection{Quotient Approach to Inner Products}

In the case where only one trajectory is available for the dynamical system, there is a limited number of options available. First, if $f$ is known to be explicitly a linear combination of the basis functions, then the weights may be determined through the factorization of the Gram matrix above. However, to make the pre-inner product a proper inner product, an appeal may be made to the quotient space, $\mathcal{F}/\mathcal{N}$, where $\mathcal{N} := \{ \eta : \mathbb{R}^n \to \mathbb{R}^n | \eta \in \mathcal{F} \text{ and } \eta\left( \gamma[0,T] \right) = \{ 0 \} \}.$ The space $\mathcal{F}/\mathcal{N}$ is an inner product space consisting of equivalence classes of functions from $\mathcal{F}$, where two members of $\mathcal{F}$ are equivalent if their difference is in $\mathcal{N}$. This is a typical construction of an inner product space from a pre-inner product space. Details can be found in standard references, such as \cite{pedersen2012analysis}.

Computations of the inner product in this form are precisely the same as those for the pre-inner product, where a member of each equivalence class is selected to represent each of $f$ and $Y_i$ for $i=1,\ldots,M$. This will motivate the algorithm for parameter identification in Section \ref{sec:systemid}.

\subsubsection{Integration Approach to Inner Products}

When a large collection of trajectories is available from a dynamical system, this collection of trajectories can be used to generate an inner product from the collection of corresponding pre-inner products. In particular, if an $(n-1)$ dimensional sub-manifold of initial points for a collection of observed trajectories is given, where there is at least one point on the submanifold, where the dynamics is nonzero, then this collection of trajectories can be leveraged to give an inner product on the collection of dynamical systems giving rise to densely defined Liouville operators.

\begin{theorem}\label{thm:inner-product}
 Let $f : \mathbb{R}^n \to \mathbb{R}^n$ be a dynamical system that gives rise to a densely defined Liouville operator over a universal RKHS of real analytic functions on $\mathbb{R}^n$, such that the operator valued kernel, $\nabla_1 \nabla_2 K(x,y)$ is universal. Let $X \subset \mathbb{R}^n$ be a compact smooth Riemann submanifold of $\mathbb{R}^n$. Suppose that $\Omega := \{ \gamma_{\xi} : [0,T] \to \mathbb{R}^n \}_{\xi \in X}$ is a collection of trajectories corresponding to $f$, such that $\gamma_{\xi}(0) = \xi \in X$, and $\Omega$ has a nonempty interior. Define the bilinear form on $\mathcal{F}$, $\langle \cdot, \cdot \rangle_{\mathcal{F},\Omega}$, as 
\[ \langle p, q \rangle_{\mathcal{F},\Omega} = \int_{X} \langle p, q \rangle_{\mathcal{F},\gamma_{\xi}} d\xi.\]
This bilinear form is an inner product.
\end{theorem}

\begin{proof}
Linearity and the semidefinite property of the bilinear form follows directly from the same properties of $\langle p, q \rangle_{\mathcal{F},\gamma_{\xi}}$ for each $\xi \in X$.

Let $\Psi = (\Psi_1,\Psi_2, \ldots)$ be a feature map for the kernel function $K$, where each $\{ \Psi_m \}_{m=1}^\infty$ forms an orthonormal basis for the RKHS. Note that the matrix $\mathbb{K}(x,y) := \nabla_1 \nabla_2 K(x,y) = \sum_{m=1}^\infty \nabla \Psi_m^T(x) \nabla \Psi_m(x)$ is an operator valued kernel for a RKHS \cite{caponnetto2008universal}. The universality of $\mathbb{K}$ yields the universality of its features \cite{caponnetto2008universal}.

Let $Y, \tilde Y \in \mathcal{F}$, and suppose that $\int_X \langle A_{Y - \tilde Y}^* \Gamma_{\gamma_\xi}, A_{Y - \tilde Y}^* \Gamma_{\gamma_\xi}\rangle d \xi= 0$. Hence, $A_{Y-\tilde Y}^* \Gamma_{\gamma_xi} = \sum_{m=1}^\infty \Psi_m(x) \int_0^T \nabla \Psi_m(\gamma_\xi(t))(Y(\gamma_\xi(t)) - \tilde Y(\gamma_\xi(t)) dt = 0$ for almost all $\xi \in X$. By the linear independence of the features, $\int_0^T \nabla \Psi_m(\gamma_\xi(t))(Y(\gamma_\xi(t)) - \tilde Y(\gamma_\xi(t)) dt = 0$ for all $m$. By the universality of $\Psi_m$, for any $\epsilon > 0$ there is a linear combination of the $\Psi_m$ that uniformly approximates $Y-\tilde Y$ over $\gamma_\xi[0,T]$ within $\epsilon$. Hence, $\int_0^T \| Y(\gamma_\xi(t)) - \tilde Y(\gamma_\xi(t) \|_2^2 dt = 0$. As $Y$ and $\tilde Y$ are both continuous functions, $Y(x) = \tilde Y(x)$ for all $x \in \gamma_\xi[0,T]$. As this holds for almost all $\xi \in X$, $Y(x) = \tilde Y(x)$ for all $x \in \Omega$ by continuity.

Since, $Y$ and $\tilde Y$ are vector valued real analytic functions over $\mathbb{R}^n$ by Theorem \ref{thm:real-analytic-symbol} and $\Omega$ has a nonempty interior, it follows that $Y = \tilde Y$ as functions over $\mathbb{R}^n$.
\end{proof}


}

\section{Parameter Identification via Occupation Kernels\label{sec:sysIDMatrix}}
{\color{black}To utilize the integral formulation of the occupation kernel determined in Section \ref{sec:occupation-estimation}, the collection of test functions leveraged in the sequel will be a finite collection of kernel functions. The method below may be seen as a weak formulation in the sense of Hilbert space inner products, where instead of determining the parameters directly from $A_f = \sum_{m=1}^M \theta_m A_{Y_m}$, the problem is resolved on test functions through an integral. Selection of kernel functions as test function allows for the analysis of various error sources using the developed framework. It should be noted that the errors may also be analyzed using different methods specific to other possible selections of test functions.} 

For a compact set $X \subset \mathbb{R}^n$, let $\{ \gamma_{j}: [0,T] \to X \}_{j=1}^N$ be a collection of trajectories satisfying the dynamics $\dot x = f(x) = \sum_{i=1}^M \theta_i Y_i(x),$ and let $\Gamma_{\gamma_j}$ be the corresponding occupation kernels inside a RKHS, $\tilde H$ of continuously differentiable functions over $X$. Suppose that $\{c_s\}_{s=1}^\infty \subset X$ is dense. {\color{black}The feature map generated by the set of test functions $ \{K(\cdot,c_s\}_{s=1}^S \} $ is $\Psi(x) = ( \tilde K(x,c_1), \ldots, \tilde K(x,c_S))^T$. The feature map yields a finite-dimensional representation of the constraints in \eqref{eq:infiniteDimensionalConstraints} as}
\begin{gather}\label{eq:constraints}
    \langle A_f \tilde K(\cdot, c_s), \Gamma_{\gamma_j} \rangle_H =\\
    \sum_{i=1}^M \theta_i \langle A_{Y_i} K(\cdot, c_s), \Gamma_{\gamma_j} \rangle_H = \tilde K(\gamma_j(T),c_s) - \tilde K(\gamma_j(0),c_s), \nonumber
\end{gather}
for each $s = 1,\ldots,\infty$ and $j = 1,\ldots, N$, which can be expressed in a matrix notation as 
\begin{gather}\label{eq:matrix-constraints}
\bold A \bold \theta = \bold K(T) - \bold K(0), \text{ where}\\
\nonumber
\bold A = \begin{pmatrix} V_{\gamma_1} \\ \vdots \\ V_{\gamma_{N}} \end{pmatrix} \in \mathbb{R}^{SN\times M}
,\\
\nonumber
\bold\theta =
\begin{pmatrix}
\theta_1 & \cdots & \theta_{M}
\end{pmatrix}^T \in \mathbb{R}^M, \text{ and}\\
\nonumber
\bold K(t) = \begin{pmatrix} \Psi(\gamma_1(t))\\ \vdots \\ \Psi(\gamma_M(t))\end{pmatrix}
\in \mathbb{R}^{SN}.
\end{gather}

Under the additional assumption of continuous differentiability of both the kernel functions and the trajectories $\{\gamma_j\}_{j=1}^M$, it can be observed through the Cauchy-Schwarz inequality that
\begin{gather*}
    |\langle A_{Y_i} \tilde K(\cdot,c_s), \hat \Gamma_{\gamma_j} \rangle_H - \langle A_{Y_i} \tilde K(\cdot,c_s), \Gamma_{\gamma_j} \rangle_H|\\
    \le \| A_{Y_i} \tilde K(\cdot,c_s) \|_H \|\hat \Gamma_{\gamma_j} - \Gamma_{\gamma_j} \|_H.
\end{gather*}
Hence, by Proposition \ref{prop:quadrature-convergence}
\begin{equation}\label{eq:sqrth}\langle A_{Y_i} \tilde K(\cdot, c_s), \hat \Gamma_{\gamma_j} \rangle_H = \langle A_{Y_i} \tilde K(\cdot, c_s), \Gamma_{\gamma_j} \rangle_H + O\left(\sqrt{h}\right),\end{equation} so that quadrature techniques can be successfully employed for estimation of the inner products contained in \eqref{eq:matrix-constraints}.

Since the matrix $\bold A$ must be numerically estimated, written as $\hat{\bold A}$, the parameter values obtained using this method are approximate, and will be represented as $\hat \theta$, obtained via \[ \hat{\theta} := (\hat{\bold A}^T \hat{\bold A})^{-1} \hat{\bold A}^T (\bold K(T) - \bold K(0)). \]
{\color{black}The complete algorithm for the system identification method is given in Algorithm \ref{alg:system-id}. Note that in the case of $N = 1$, $A^TA = V_{\gamma_1}^T V_{\gamma_1}$ is the Gram matrix given in Section \ref{sec:testfree}, with respect to the kernel function $K(x,y) = \sum_{s=1}^S \tilde K(x,c_s)\tilde K(y,c_s)$. 

\begin{algorithm}
\caption{Pseudocode for the system identification routine of Section \ref{sec:systemid}. In the description some quantities are left in their analytic form, such as the integral of line $7$. The choice of quadrature routine can have a significant impact on the overall results, and it is advised that a high accuracy method is leveraged in practice.}
\label{alg:system-id}
\begin{algorithmic}[1]
\STATE{Input: Trajectories $\{ \gamma_{j} :[0,T] \to \mathbb{R}^n \}_{j=1}^M$, Centers $\{c_s\}_{s=1}^S$, and basis $\{ Y_i\}_{i=1}^N$}
\STATE{Initialize the empty matrix $\bold A$ and empty vector $\bold b$}
\FOR{j'=1:M}
    \STATE{Initialize $S \times N$ matrix $\bold A_{j'}$ and the length $S$ vector $\bold b_{j'}$}
    \FOR{s'=1 to S}
        \FOR{i'=1 to N}
            \STATE{Assign the value of the integral $\int_0^T \nabla \tilde K(\gamma_{j'}(t),c_{s'}) Y_i(\gamma_{j'}(t))dt$ to the $(s',i')$ entry of $\bold A_{j'}$.}
            \STATE{Assign the value $\tilde K(\gamma_{j'}(T),c_{s'}) - \tilde K(\gamma_{j'}(0),c_{s'})$ to the $s'$ entry of $\bold b_{j'}$.}
        \ENDFOR
    \ENDFOR
    \STATE{Append $\bold A_{j'}$ and $\bold b_{j'}$ to $\bold A$ and $\bold b$ respectively.}
\ENDFOR
\RETURN{$\theta$ as $(\bold A^T \bold A)^{-1} \bold A^T \bold b$.}
\end{algorithmic}
\end{algorithm}

{\color{black}
\subsection{A note on the independence of the algorithm}

As was noted above, the constraints for the parameter identification routine can be established independent of the framework of Section \ref{sec:testfree}. However, the utilization of the Hilbert space framework for error estimates requires one additional assumption on the basis functions $Y_i$.

\begin{assumption}\label{ass:common-domain}
Given a RKHS, $H$, over a set $X$, each of the operators, $A_{Y_i}: \mathcal{D}(A_{Y_i}) \to H$ are densely defined. Moreover, $\cap_{i=1}^M \mathcal{D}(A_{Y_i})$ is dense in $H$. That is, the operators $A_{Y_1},\ldots, A_{Y_M}$ have a common dense domain.
\end{assumption}


Assumption \ref{ass:common-domain} ensures the validity of decomposing $A_f$ into a linear combination of densely defined Liouville operators, $\{ A_{Y_i}\}_{i=1}^M$. Liouville operators are closely connected to densely defined multiplication operators (c.f. \cite{rosenfeld2015densely,rosenfeld2015introducing,rosenfeld2016sarason,sarason2008unbounded}), and the unavailability of complete classifications of densely defined multiplication operators over many RKHSs indicates that characterizing the necessary and sufficient conditions that a dynamical system must meet to allow a Liouville operator to be densely defined may be an intractable problem in many cases. However, sufficient conditions can certainly be established. In particular, Assumption \ref{ass:common-domain} is borne out through examination of the exponential dot product kernel, where a polynomial function $f$ may be decomposed into linear combinations of polynomials, each of which has a corresponding Liouville operator containing polynomials inside of its domain. More sophisticated examples of decompositions can be expressed and treated individually.

For other collections of test functions, where methods outside of the Hilbert space framework can be employed for error analysis, Assumption \ref{ass:common-domain} is not necessary.
}
}

\section{Impact of Signal Noise on Samples}\label{sec:noise}

{\color{black}In this section we consider a differentiable test function, $g : \mathbb{R}^n \to \mathbb{R}^n$ that comprises one of the features of the kernels mentioned in \ref{sec:testfree}. This test function could come from an extant kernel function, such as the Gaussian RBF kernel, or it can be part of a collection leveraged to define a new kernel for the purpose of system identification. As the features in the implementation described in \ref{sec:sysIDMatrix} are derived from kernel functions themselves, this provides an opportunity to demonstrate analysis techniques within a RKHS using occupation kernels. Many of the error bounds developed here can also be developed using bounds on the gradients of the test functions over compact sets, but will instead be demonstrated using kernel techniques. The kernels used as features will be written as $\tilde K$ and the occupation kernels will be written as $\tilde \Gamma_{\gamma}$ to distinguish them from the kernel defining the inner product on the dynamical systems. In this case, the relationship between the two kernels is given via $K(x,y) = \sum_{i=1}^M \tilde K(x,c_i) \tilde K(x,c_j)$.}

\subsection{Measurement Noise}

An immediate advantage evident in the usage of occupation kernel methods for system identification over that of methods employing numerical derivative estimates is a certain robustness to noise, which is demonstrated in Figure \ref{fig:noisy}. Indeed, signal noise added to a signal requires sophisticated filtering techniques to allow for reasonable numerical derivative estimates \cite{brunton2016discovering}. On the other hand, normally distributed white noise has a smaller effect on integration based methods, since peaks in the noise are infinitesimally small and carry less weight through the integration process.

In the context of occupation kernel based methods, a sample for the system identification method takes the form
\[ \langle A_{Y_l} \tilde K(\cdot,c_i), \tilde \Gamma_{\gamma_j}\rangle_H = \int_0^T \nabla \tilde K(\gamma_{j}(t),c_i) Y_{l}(\gamma_{j}(t)) dt \] as in \eqref{eq:matrix-constraints}. Let $\epsilon:[0,T]\to \mathbb{R}^n$ be a disturbance term acting as signal noise, then the noise corrupted sample may be expressed as 
\[ \langle A_{Y_l} \tilde K(\cdot,c_i), \tilde \Gamma_{\gamma_j+\epsilon}\rangle_H = \int_0^T \nabla \tilde K(\gamma_{j}(t)+\epsilon(t),c_i) Y_{l}(\gamma_{j}(t)+\epsilon(t)) dt. \] The following theorem provides a bound on the difference between the corrupted and uncorrupted signals.

\begin{theorem}\label{thm:error}
Suppose that $H$ is a RKHS consisting of twice continuously differentiable functions and $Y_l$ is continuously differentiable for each $l$, then the error introduced by a bounded zero mean disturbance\footnote{$L^2([0,T],\mathbb{R}^n)$ denotes the Lebesgue space of functions $g:[0,T] \to \mathbb{R}^n$ such that $\int_0^T \| g(t) \|_2^2 dt < \infty$.}, $\epsilon \in L^2([0,T],\mathbb{R}^n)$, is $O(T\cdot \sigma(\epsilon))$ where $\sigma(\epsilon)$ is the standard deviation of $\epsilon$ with respect to the uniform probability distribution over $[0,T]$.
\end{theorem}

\begin{proof}
Consider,
\begin{gather}
    \left| \langle A_{Y_l} \tilde K(\cdot,c_i), \tilde \Gamma_{\gamma_j}\rangle_H - \langle A_{Y_l} \tilde K(\cdot,c_i), \tilde \Gamma_{\gamma_j+\epsilon}\rangle_H\right| \nonumber\\
    \label{eq:corruptionequation}
    =\left|\int_0^T \nabla \tilde K(\gamma_{j}(t),c_i) Y_{l}(\gamma_{j}(t)) - \nabla \tilde K(\gamma_{j}(t)+\epsilon(t),c_i) Y_{l}(\gamma_{j}(t)+\epsilon(t)) dt.\right|.
\end{gather}

By the hypothesis, $\nabla \tilde K(\cdot,c_i) Y_l(\cdot)$ is continuously differentiable. Hence, if given $R>0$, $B_R(0)$ is a ball containing the image of $\gamma_j$ and $\tilde R > 0$ is a bound on the disturbance, $\epsilon$, then $\overline{B_{R+\tilde R}(0)}$ is a compact region containing the image of $\gamma_{j}+\epsilon$. Let $\tilde S$ be an upper bound on the derivative of $\nabla \tilde K(\cdot,c_i) Y_l(\cdot)$ over $\overline{B_{R+\tilde R}(0)}$. Thus, by the mean value inequality, \eqref{eq:corruptionequation} may be bounded as
\begin{gather*}
    \le  \tilde S \int_0^T \|\gamma_{j}(t) + \epsilon(t) - \gamma_{j}(t)\|_2 dt
    \le \tilde S \sqrt{\int_0^T 1 dt} \sqrt{ \int_0^T \| \epsilon(t) \|_2^2 dt}\\
    \le \tilde S \sqrt{T} \sqrt{ \frac{T}{T}\int_0^T \| \epsilon(t) \|_2^2 dt}
    \le \tilde S T \sqrt{ \frac{1}{T} \int_0^T \| \epsilon(t) \|_2^2 dt} = \tilde S T \sigma(\epsilon).
\end{gather*}
Hence, the big-oh estimate is established.
\end{proof}

Note that the above theorem may be modified to accommodate a possibly unbounded disturbance in $L^2([0,T],\mathbb{R}^n)$ when $Y_l$ and $\nabla \tilde  K(\cdot,c_i)$ have bounded derivatives and Jacobians respectively.

\begin{figure}
    \centering
    \includegraphics[scale=0.2]{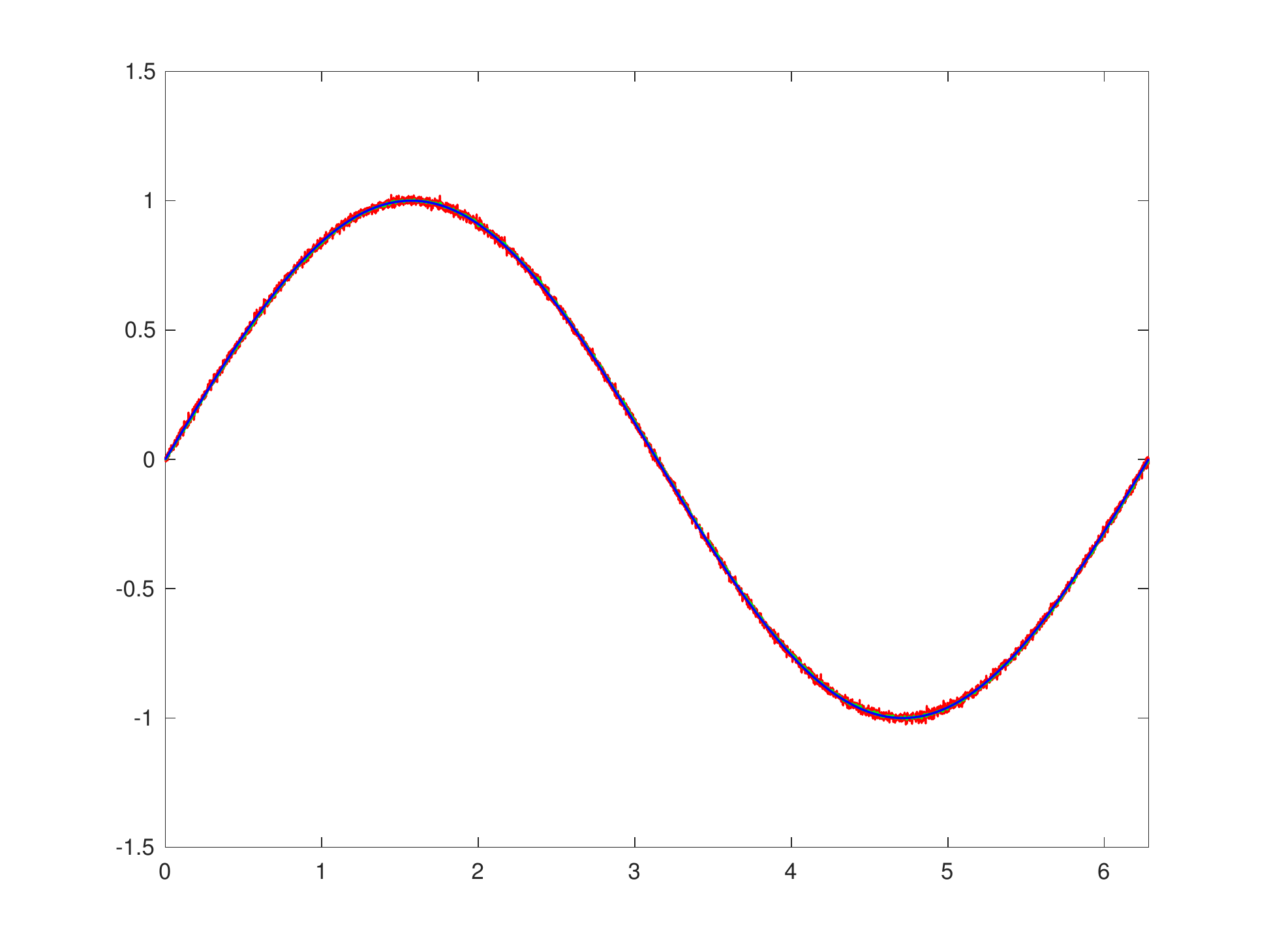}%
    \includegraphics[scale=0.2]{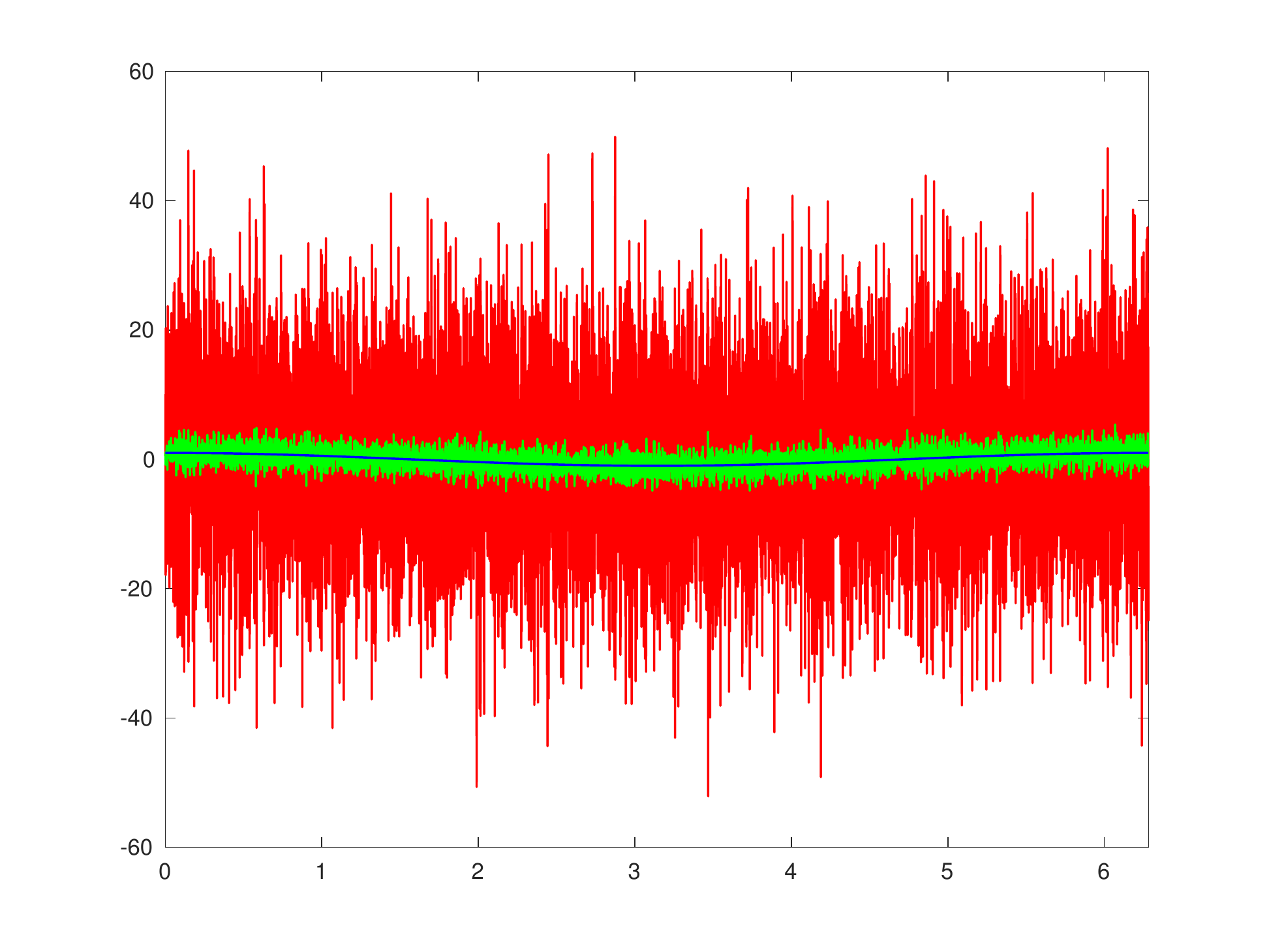}%
    \includegraphics[scale=0.2]{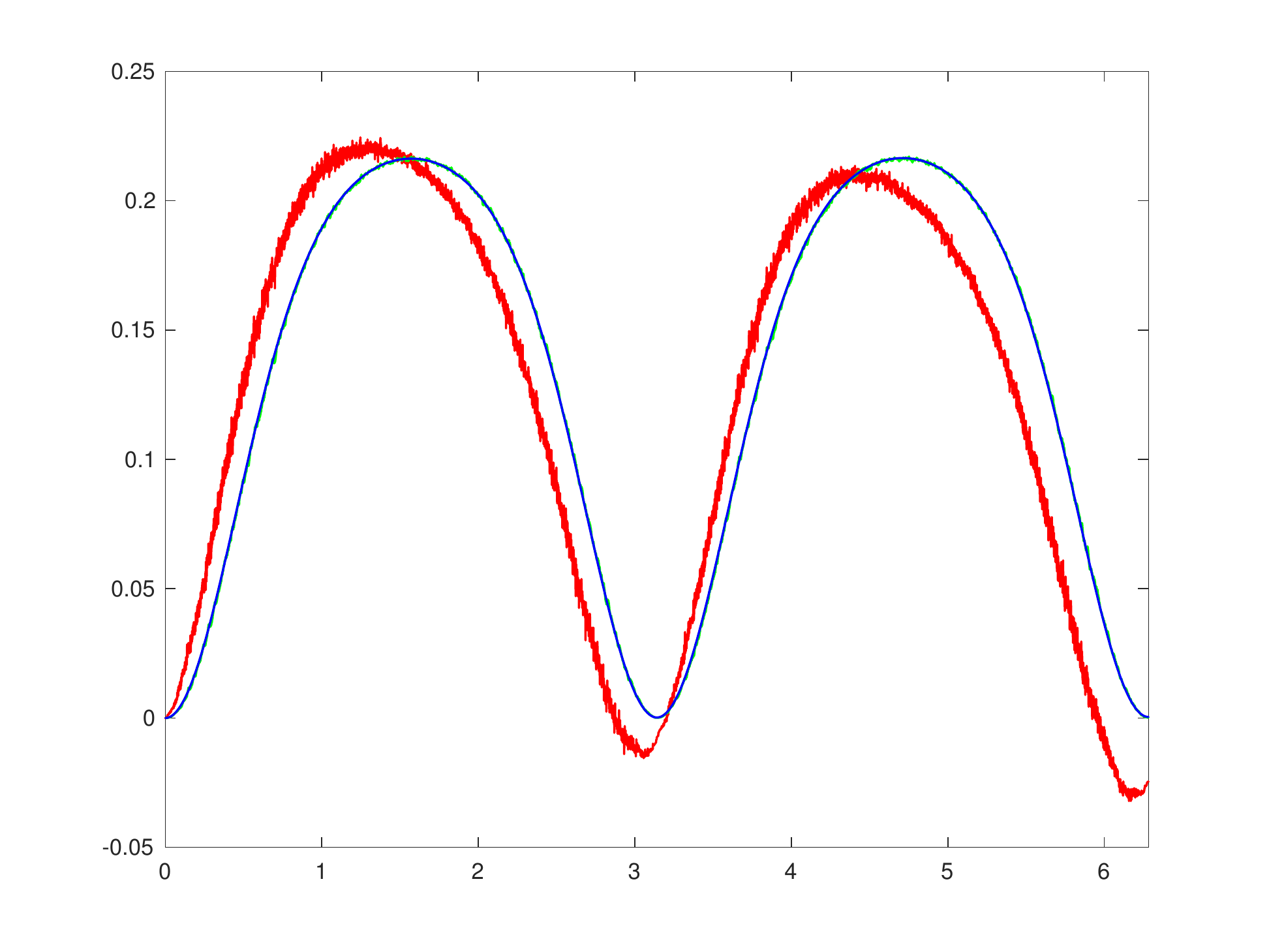}
    \caption{\label{fig:noisy}This figure contrasts samples of a dynamical system from a single trajectory according to numerical derivatives versus samples according to occupation kernels. The trajectory, $\sin(t)$ over $[0,2\pi]$, is shown in the left figure (blue) along with a noise corrupted trajectory (red) and a 10 point moving average filter of the noise corrupted trajectory trajectory (green). The disturbance, $\epsilon$, is normally distributed white noise with mean zero and standard deviation $0.01$.  The center figure shows numerical derivatives obtained from each trajectory, and the right figure shows samples obtained using occupation kernels.  The kernel function used for the right figure is the Gaussian RBF with kernel width $0.5$, and in place of a basis function the numerical derivative estimate of $\dot y$ is used for a worst case example.  The center figure demonstrates a very large error in the estimation of the derivative while using numerical techniques on a noise corrupted trajectory even after the application of a moving average filter. The right figure demonstrates a much smaller error, and the red trajectory validates the $O(T\cdot \sigma(\epsilon))$ estimate of Theorem \ref{thm:error}. This figure demonstrates the occupation kernel samples' robustness to noise, where even when an unfiltered noisy signal is used, there is a very small error in the sample.}
    \label{fig:my_label}
\end{figure}

\subsection{Process Noise}
The above discussion concerns measurement or signal noise for the system. Another common source of noise is process noise, where a state dependent disturbance, $\eta :\mathbb{R}^n \to \mathbb{R}^n$, impacts the dynamics directly as $\dot x = f(x) + \eta(x).$ In this case, there is an impact on the occupation kernel as well as the state variable and Liouville operator. Supposing that $z:[0,T] \to \mathbb{R}^n$ is the unperturbed state trajectory satisfying the same initial condition, $z(0) = x(0)$, then 
\begin{equation}\label{eq:process-state-ineq}|z(t) - x(t)| \le e^{Lt} \int_0^t e^{-L\tau} \| \eta(x(\tau)) \|_2 d \tau.\end{equation}

Leveraging the argument in Theorem \ref{thm:error}, the norm difference between the respective occupation kernels for $x(\cdot)$ and $z(\cdot)$ is bounded as
\begin{equation}\label{eq:occupation-distance-process}
\| \tilde \Gamma_{x(\cdot)} - \tilde \Gamma_{z(\cdot)}\|_H^2 \le \tilde S^2 T \int_0^T e^{2Lt} \left(\int_0^t e^{-L\tau} \| \eta(x(\tau)) \|_2 d\tau\right)^2 dt,
\end{equation}
and also note that\footnote{For a matrix $A$, the notation $\|A\|_F$ denotes the Frobenius norm of $A$.} $\| A_\eta \tilde K(\cdot,c) \|_H^2 \le \|\nabla_1 \nabla_2 K(c,c)\|_{F} \| \eta(c) \|_2^2$ and $\| A_f \tilde K(\cdot,c) \|_H^2 \le \|\nabla_1 \nabla_2 \tilde K(c,c)\|_F \| f(c) \|_2^2$.

Assuming that $g \in \mathcal{D}(A_{f+\eta}) \cap \mathcal{D}(A_{\eta}) \cap \mathcal{D}(A_{f})$, then \begin{gather*} |\langle A_{f+\eta} g, \Gamma_{x(\cdot)}\rangle_H - \langle A_{f} g, \Gamma_{z(\cdot)} \rangle_H |\\ \le |\langle A_{f+\eta} g, \Gamma_{x(\cdot)}\rangle_H - \langle A_{f} g, \Gamma_{x(\cdot)}\rangle|_H + |\langle A_{f} g, \Gamma_{x(\cdot)}\rangle_H - \langle A_{f} g, \Gamma_{z(\cdot)} \rangle_H |\\
    \le \| A_\eta g \|_H \| \Gamma_{x(\cdot)} \|_H + \| A_f g \|_H \| \Gamma_{x(\cdot)} - \Gamma_{z(\cdot)}\|_H,
    \end{gather*}
and taking $g = \tilde K(\cdot,c)$ for some $c \in \mathbb{R}^n$ the following inequality holds
\begin{gather*} |\langle A_{f+\eta} \tilde K(\cdot,c), \tilde \Gamma_{x(\cdot)}\rangle_H - \langle A_{f} \tilde K(\cdot,c), \tilde \Gamma_{z(\cdot)} \rangle_H | \le \sqrt{\|\nabla_1 \nabla_2 \tilde K(c,c)\|_F \| \eta(c) \|_2^2} \| \tilde \Gamma_{x(\cdot)} \|_H\\
    + \sqrt{\|\nabla_1 \nabla_2 \tilde K(c,c)\|_F \| f(c) \|_2^2} \sqrt{\tilde S^2 T \int_0^T e^{2Lt} \left(\int_0^t e^{-L\tau} \| \eta(x(\tau)) \|_2 d\tau\right)^2 dt}.
    \end{gather*}
Thus, the following theorem has been demonstrated,
\begin{theorem}\label{thm:process-error}
If the system is subjected to a process noise, $\eta$, that is bounded over a compact set containing the trajectories, $\Omega$, which yields a densely defined Liouville operator, then error induced in the samples may be bounded as
\begin{gather*}\sqrt{\|\nabla_1 \nabla_2 \tilde K(c,c)\|_F \| \eta(c) \|_2^2} \| \tilde \Gamma_{x(\cdot)} \|_H\\
    + \sqrt{\|\nabla_1 \nabla_2 \tilde K(c,c)\|_F \| f(c) \|_2^2} \sqrt{\tilde S^2 T \int_0^T e^{2Lt} \left(\int_0^t e^{-L\tau} \| \eta(x(\tau)) \|_2 d\tau\right)^2 dt},  \end{gather*} which is $O(\sup_{x \in \Omega}\|\eta(x)\|_2).$
\end{theorem}

The error induced by the process noise is ultimately larger than that induced by measurement noise. However, process noise with small bounds leads to small errors in the samples.

\section{Incorporating Streaming Data}\label{sec:streaming}

The principle observation of this section is that the matrix, $\bold A$, given in \eqref{eq:matrix-constraints} changes continuously with respect to the time variable, $T$. That is, if $\gamma:[0,T] \to \mathbb{R}^n$ is a continuous trajectory and $\tilde \Gamma_{\gamma,\tau} := \int_0^\tau \tilde K(\cdot,\gamma(t))dt$, then the matrix
\begin{equation}\label{eq:time-varying-A}\bold A(\tau) := V_{\left.\gamma\right|_{[0,\tau]}}\end{equation}
is continuous with respect to $\tau$. This continuity may be demonstrated directly from the integral representations of the inner products contained within $\bold A(\tau)$. However, it is useful to document the following lemma.

\begin{lemma}\label{lem:time-continuity}
Suppose that $H$ is a RKHS of continuous functions over a compact set $X \subset \mathbb{R}^n$, and let $\gamma:[0,T] \to X$ be a continuous trajectory. The mapping $t \mapsto \Gamma_{\gamma,t}$ is continuous in the Hilbert space norm of $H$.
\end{lemma}

\begin{proof}
Let $t_1 < t_2$ and consider \[\| \Gamma_{\gamma,t_2} - \Gamma_{\gamma,t_1} \|_H^2 = \langle \Gamma_{\gamma,t_2}, \Gamma_{\gamma,t_2} \rangle_H + \langle \Gamma_{\gamma,t_1},\Gamma_{\gamma,t_1}\rangle_H - 2 \langle \Gamma_{\gamma,t_1},\Gamma_{\gamma,t_2}\rangle_H.\]

Now compare the inner products, $\langle \Gamma_{\gamma,t_2}, \Gamma_{\gamma,t_2} \rangle_H$ and $\langle \Gamma_{\gamma,t_1},\Gamma_{\gamma,t_2}\rangle_H$. In particular, observe
\begin{gather*}
    \left|\langle \Gamma_{\gamma,t_2}, \Gamma_{\gamma,t_2} \rangle_H-\langle \Gamma_{\gamma,t_1},\Gamma_{\gamma,t_2}\rangle_H\right|\\
    = \left|\int_0^{t_2}\left( \int_0^{t_2} K(\gamma(t),\gamma(s)) dt -\int_0^{t_1} K(\gamma(t),\gamma(s)) dt  \right)ds\right|\\
    = \left|\int_0^{t_2}\left( \int_{t_1}^{t_2} K(\gamma(t),\gamma(s)) dt \right)ds\right| \le t_2(t_2-t_1) \sup_{x,y \in X} |K(x,y)|.
\end{gather*}
Hence, as $t_1 \to t_2$ the difference $\left|\langle \Gamma_{\gamma,t_2}, \Gamma_{\gamma,t_2} \rangle_H-\langle \Gamma_{\gamma,t_1},\Gamma_{\gamma,t_2}\rangle_H\right| \to 0.$ Similarly, it can be shown that $\left|\langle \Gamma_{\gamma,t_2}, \Gamma_{\gamma,t_2} \rangle_H-\langle \Gamma_{\gamma,t_1},\Gamma_{\gamma,t_1}\rangle_H\right| \to 0$ as well. Thus, continuity is established.
\end{proof}

Writing $\bold A(t_2) - \bold A(t_1) = \begin{pmatrix}
\langle A_{Y_i} \tilde K(\cdot, c_{n_{j,1}}),\tilde \Gamma_{\gamma_{n_{j,2}},t_2}-\tilde \Gamma_{\gamma_{n_{j,2}},t_1} \rangle_H
\end{pmatrix}_{j=1,i=1}^{j=S,i=M},$ the continuity of $\bold A(\tau)$ follows from a term by term application of the Cauchy Schwarz inequality and Lemma \ref{lem:time-continuity}. Hence, for a small change from $t_1$ to $t_2$, there is a correspondingly small change between $\bold A(t_1)$ and $\bold A(t_2)$. Consequently, if $\bold A(t)$ has full column rank, then $\bold A(\tau)^T \bold A(\tau)$ is continuously invertible for $\tau$ in a neighborhood of $t$ \cite[Excercise 4.1.6]{pedersen2012analysis}.

Consequently, as the calculated $\hat \theta$ is an approximation of the true parameter vector $\theta$, the accuracy of which depends on the quality of the data collected in $\hat{\bold A}$, $\hat \theta$ may be viewed as a function of time when updated according to $\bold A(t)$.

\begin{lemma}\label{lem:parameter-continuity}
Suppose that $H$ is a RKHS of continuously differentiable functions, and let $\bold A(t)$ and $\bold K(t)$ be as in \eqref{eq:time-varying-A} and \eqref{eq:constraints} respectively. Let $\hat \theta(t)$ be the solution to $\bold A(t) \hat \theta(t) = \bold K(t)-\bold K(0)$. If $t \in [0,T]$ is such that $\bold A(t)$ has full column rank, then $\hat \theta(\tau)$ is continuous for $\tau$ in a neighborhood of $t$.
\end{lemma}

\begin{proof}
The continuity of $\hat \theta(t)$ follows from the discussion preceeding Lemma \ref{lem:parameter-continuity} and the observation that $ \hat \theta(t) = (\bold A(t)^T \bold A(t))^{-1} \bold A(t)^T (\bold K(t) - \bold K(0)).$
\end{proof}

The update of the parameter $\hat \theta(\cdot)$ between two time instances, $0 < t_1 < t_2 < T$, may be expressed as follows
\begin{gather*}
    \hat \theta(t_2) = (\bold A(t_2)^T \bold A(t_2))^{-1}\bold A(t_2)^T \left( \bold A(t_1) \hat \theta(t_1) + (\bold K(t_2)-\bold K(t_1) \right)\\
    = (\bold A(t_2)^T \bold A(t_2))^{-1}\bold A(t_2)^T\bold A(t_1) \hat \theta(t_1) + (\bold A(t_2)^T \bold A(t_2))^{-1}\bold A(t_2)^T(\bold K(t_2)-\bold K(t_1)),
\end{gather*}
which is cumbersome to implement numerically. However, the continuity of $\hat \theta(\cdot)$ motivates gradient based update laws.

\begin{theorem}[A Gradient Chase Theorem]
Let $\tau > 0$, and suppose $\bold A(t)$ is full rank for $t > \tau$ and the time varying minimizer of \begin{equation}\label{eq:objectivefunction}F(\theta,t) = \frac12 \theta^T \bold A(t)^T \bold A(t) \theta - \theta^T \bold A(t)^T (\bold K(t) - \bold K(0)),\end{equation} given as $\hat \theta^*(t)$, is Lipshitz continuous. If the gradient descent algorithm is applied at a fixed interval, $h > 0$, then the iterated numerical sequence, $\hat \theta_k$, converges exponentially to a fixed error.
\end{theorem}

\begin{proof}This follows directly from \cite[Theorem 1]{simonetto2016class} and the above discussion.
\end{proof}

Several variations of the above theorem can be realized with the same conclusions. For example, if the parameters are not constant, but are time varying, then \eqref{eq:objectivefunction} can be adjusted as \begin{equation}\label{eq:objectivefunction2}F(\theta,t) = \frac12 \theta^T \bold B(t)^T \bold B(t) \theta - \theta^T \bold B(t)^T (K(t) - K(t-s)),\end{equation} for some $s > 0$, and $\bold B(t) := \bold A(t) - \bold A(t-s)$. The principle difference in the implementation of \eqref{eq:objectivefunction} and \eqref{eq:objectivefunction2} is that \eqref{eq:objectivefunction} may be progressed in time while managing only one matrix $\bold A(t)$ by adding only the most recent integral segments, while \eqref{eq:objectivefunction2} requires the maintenance of the history of all of the matrix elements.

{\color{black}
\section{A Stable Variation of the SINDy Algorithm}

The constraints established in this manuscript for the determination of the parameters in a parameter identification routine can be readily employed in a sparse learning framework, such as the SINDy algorithm in \cite{brunton2016discovering}. The key difference between \cite{brunton2016discovering} and the present manuscript is that the SINDy algorithm obtains samples of the dynamical system through numerical derivatives of the system state. These numerical derivatives must be followed by extensive filtering to ensure a tractable algorithm. After the essential matrices are created, the parameters for the system are determined through an $\ell^1$ regularization or another sparsity promoting regression scheme to obtain a small number of nonzero parameters. The system may then be reduced to only a small number of basis functions and further refinements to the parameters may be performed on the reduced system. In an identical fashion, sparsity promoting algorithms, such as LASSO, may be employed to give a sparse solution to \eqref{eq:matrix-constraints} to assist with dimensionality reduction. However, the key advantage realized through the present approach is that integration is more robust to sensor noise than numerical differentation as demonstrated in the previous sections.
}

\section{Numerical Experiments}\label{sec:numerical}

Two simulated systems were examined to evaluate the system identification method of Section \ref{sec:systemid}, and one system arising from real world data was also examined. For each simulated system, the trajectories were generated using the Runge-Kutta 4 algorithm with step size $h=0.001$. On each system several different experiments were performed to evaluate the effects of various parameters, such as the kernel width, the selection of kernel, the numerical integration method, and the number of trajectories utilized. For each system, the centers of the kernel were kept constant throughout the experiments. The dynamics in each example are treated as unknown and are parameterized with respect to the collection monomials of degree up to two. Unless otherwise noted, the matrix $\bold A$ in \eqref{eq:matrix-constraints} for each experiment was computed using Simpson's Rule for numerical integration.

\begin{system}\label{sys:one}The first dynamical system is sourced from a collection of benchmark examples for the formal verification community presented in \cite{sogokon2016nonlinear}. The two dimensional dynamics are given as \begin{equation}\label{eq:system1dynamics}
\dot x = f(x) = \begin{pmatrix} 2x_1-x_1x_2\\2x_1^2-x_2\end{pmatrix}.
\end{equation}
Twenty five trajectories were generated for this system over the time interval $[0,1]$ and the initial points were selected from the rectangle $[-0.5,0.5]\times[-2.5,-1.5]$ through a lattice with width $0.25$. The collection of trajectories are presented in Figure \ref{fig:system1}.

The centers for the kernel functions for System \ref{sys:one} were selected from a lattice of width $1$ over $[-3,3]\times[-3,5]$.
\end{system}

\begin{experiment}\label{exp:one}
The first experiment examines the error committed in the parameter estimation by varying the number of trajectories used in the system identification method of Section \ref{sec:systemid}. In this experiment two kernel functions were used; the Gaussian RBFs and the Exponential Dot Product Kernels. The Gaussian RBFs were used with kernel width $\mu = 10$, and the Exponential Dot Product Kernels used parameter $\mu = 1/25$. The results of Experiment \ref{exp:one} may be observed in Figure \ref{fig:varynumber}.
\end{experiment}

\begin{experiment}\label{exp:two}
The second experiment explores the effect of the kernel width, $\mu$, on the parameter estimation when using the Gaussian RBF in the system identification routine on System \ref{sys:one}. The results of Experiment \ref{exp:two} can be observed in Figure \ref{fig:varymu}.
\end{experiment}

\begin{system}\label{sys:lorenz}
The second system is the three dimensional Lorenz system \cite{perko2013differential,brunton2016discovering},
\begin{equation}\label{eq:lorenz}
    \dot x = f(x) = 
    \begin{pmatrix}
    \sigma(x_2-x_1)\\
    x_1(\rho-x_3)-x_2\\
    x_1x_2 - \beta x_3
    \end{pmatrix}.
\end{equation}
Following \cite{brunton2016discovering} a single trajectory was generated over the time interval $[0,100]$ where $\sigma = 10$, $\beta = 8/3$, $\rho = 28$, and the initial condition was given as $x_0=(-8,7,27)^T$. The plot of this trajectory is given in Figure \ref{fig:system2}.

The centers for System \ref{sys:lorenz} were obstained from a lattice with width $10$ within $[-20,20]\times[-50,50]\times[-20,50]$.
\end{system}

\begin{experiment}\label{exp:three}
This experiment investigates the error contribution committed by the use of different numerical integration schemes. In this setting System \ref{sys:lorenz} was identified using the Gaussian RBFs with kernel width $\mu = 10$. The results are displayed in Table \ref{tab:numschemes}.
\end{experiment}

\begin{table}[]
\centering
\begin{tabular}{lcr}
{\ul \textbf{Numerical Method}} & {\ul \textbf{Convergence Order}} &{\ul \textbf{Error $\|\theta-\hat\theta\|_2$}} \\
Right Hand Rule & $O(h)$                & 2.1696e+0                                   \\
Trapezoid Rule &  $O(h^2)$                & 3.8136e-2                                   \\
Simpson's Rule &  $O(h^4)$               & 7.6920e-5                                   
\end{tabular}
\caption{This table presents a comparison between the errors in parameter estimation based on the selection of typical numerical integration schemes for the system identification routine for System \ref{sys:lorenz}. Each numerical integration scheme is listed along with the convergence rate of the algorithm. Of the three routines, the Simpson's rule demonstrates the strongest performance. The step-size was kept consistent between each experiment at $h=0.001$.}
\label{tab:numschemes}
\end{table}

\begin{experiment}\label{exp:four}
This experiment is motivated by the method used to generation of the trajectory data. Runge-Kutta 4 has a high rate of convergence. However, as with any time-stepping method the global error bound is proportional $e^{LT}$ where $L$ is the Lipschitz constant of the dynamics \cite{atkinson2008introduction}. As such, the accumulated global error could be large in the long term evaluation of the trajectory of System \ref{sys:lorenz}. Experiment \ref{exp:four} investigates the effect on the error when the trajectory of System \ref{sys:lorenz} is segmented into smaller trajectories. Each smaller trajectory is then treated as a new initial value problem with a smaller time horizon and thus a hypothetically smaller global error. Here the Gaussian RBF was leveraged in the system identification algorithm of Section \ref{sec:systemid} with kernel width $\mu = 10$.
\end{experiment}

\begin{experiment}\label{exp:five}
This experiment introduces zero mean normally distributed white noise with standard deviation of $0.01$ to System \ref{sys:one} using the same parameters as in Experiment \ref{exp:one}. The system identification method was attempted on the noise corrupted trajectories as well as the corrupted trajectories treated with a $20$ point moving average filter. The results of the parameter estimates obtained for this experiment are shown in Table \ref{tab:experiments}.

{\color{black}The experiment also includes a set of a thousand Monte-Carlo trials that compare the results of the developed occupation kernel system identification method with the more straightforward ILS approach described in \eqref{eq:ILS}. Zero mean Gaussian white noise with a standard deviation of $0.01$ is added to $20$ trajectories of the Lorenz system, segmented from the trajectory in Experiment \ref{exp:three}. Monomials in $3$ variables, up to order $3$, were utilized as basis functions to yield a set of $60$ parameters to be estimated. A set of $180$ Gaussian RBF kernels with $\mu = 20^2/3 $ are used to implement the occupation kernel method. The norm of the error, $\Vert \theta - \hat\theta \Vert$, is used in Figure \ref{fig:MonteCarloILSComparison} as a metric for comparison.}

\begin{table}[]
\centering
\begin{tabular}{rr}
{\ul \textbf{Number of Segments}} &{\ul \textbf{Error $\|\theta-\hat\theta\|_2$}} \\
1                & 7.6920e-5                                    \\
10                  & 5.2175e-6                                    \\
100                & 5.8506e-6                                   
\end{tabular}
\caption{This table contrasts the parameter estimation errors committed by the system identification routine applied to System \ref{sys:lorenz} when the single trajectory is segmented into smaller trajectories. It can be observed that an order of mangnitude improvement was realized when the single trajectory was segmented into $10$ and $100$ trajectories. However, there was no improvement in the error when using $100$ segments over $10$ segments.}
\label{tab:segmented}
\end{table}
\end{experiment}
{\color{black}
\begin{experiment}
This experiment leveraged the constraints of \eqref{eq:constraints} in combination with the LASSO algorithm ($\ell^1$ regularization) for System \ref{sys:one}, where a monomial basis was selected to consist of all monomials of degree $5$ or less. Of the $42$ basis functions leveraged in this experiment, the LASSO algorithm selected only $6$ nonzero weights. These nonzero components corresponded to a collection of basis functions, where the true basis was a proper subset. The parameter values deviated significantly from those of the true parameters, however with the reduced system, the weights can be determined using the pseudo-inverse as in the rest of the manuscript. This method aligns with the SINDy algorithm of \cite{brunton2016discovering}.
\end{experiment}
}

\begin{experiment}\label{exp:six}
This experiment applies the developed technique to identify the Electro-Mechanical Positioning System (EMPS) from the nonlinear system identification benchmarks archive \cite{SCC.Janot.Gautier.ea2019}. The system is a controlled system of the form
\[
    \dot{x}_{1} = x_{2},\quad\dot{x}_{2} = \begin{pmatrix} \tau&-x_2&-\mathrm{sign}\left(x_2\right)&-1\end{pmatrix}\theta
\]
with $\theta\in\mathbb{R}^4$. Due to presence of the controller $\tau$, this is a time-varying system. To reformulate the problem in terms of an autonomous system, time is augmented to the state vector to get a system of the form
\[
    \begin{pmatrix}
        \dot{x}_{1}\\\dot{x}_{2}\\\dot{x}_{3}
    \end{pmatrix}=\begin{pmatrix}
        x_{2}\\0\\1
    \end{pmatrix}+\begin{pmatrix}
        0&0&0&0\\\tau(x_3)&-x_2&-\mathrm{sign}\left(x_2\right)&-1\\0&0&0&0
    \end{pmatrix}\theta = h(x)+\sum_{i=1}^{i=4} \theta_i Y_i(x)
\]
The system identification is then carried out using a slight modification of the procedure in Section \ref{sec:sysIDMatrix} to accommodate the known part $h(x) = \begin{pmatrix}x_{2}&0&1\end{pmatrix}^{T}$:
\[ 
    \hat \theta := (\hat{\bold A}^T \hat{\bold A})^{-1} \hat{\bold A}^T (\bold K(T) - \bold K(0) - \bold B),\quad \bold B = \begin{pmatrix}
    \langle A_{h} K(\cdot, c_{n_{j,1}}), \Gamma_{\gamma_{n_{j,2}}} \rangle_H
    \end{pmatrix}_{j=1}^{j=SN} \in \mathbb{R}^{SN}.
\]
The results of the parameter estimates obtained for this experiment using the validation data provided in \cite{SCC.Janot.Gautier.ea2019} are shown in Table \ref{tab:EMPSexperiments} and Figure \ref{fig:EMPS}.

\begin{table}[]
\centering
\begin{tabular}{rrr}
{\ul \textbf{Signal}} & {\ul \textbf{Benchmark}} & {\ul \textbf{Occupation Kernels}} \\
Position     & 8.025e-3 & 8.038e-3         \\
Velocity     & 4.789e-1 & 4.729e-1         \\
Acceleration & 1.982e+1 & 1.954e+1         \\
Force        & 8.941e+0 & 9.046e+0        
\end{tabular}
\caption{Comparison of percentage errors committed by the developed method and the IDIM-LS method from \cite{SCC.Janot.Gautier.ea2019}. A model that uses the identified parameters is simulated forward in time to obtain the simulated position, velocity, force, and acceleration trajectories using the same controller that was used to collect the validation data. The trajectories are then compared against the validation data, where velocity and acceleration are obtained through filtered central difference differentiation.\label{tab:EMPSexperiments}}
\end{table}
\begin{figure}
    \centering
    \includegraphics[width=0.8\columnwidth]{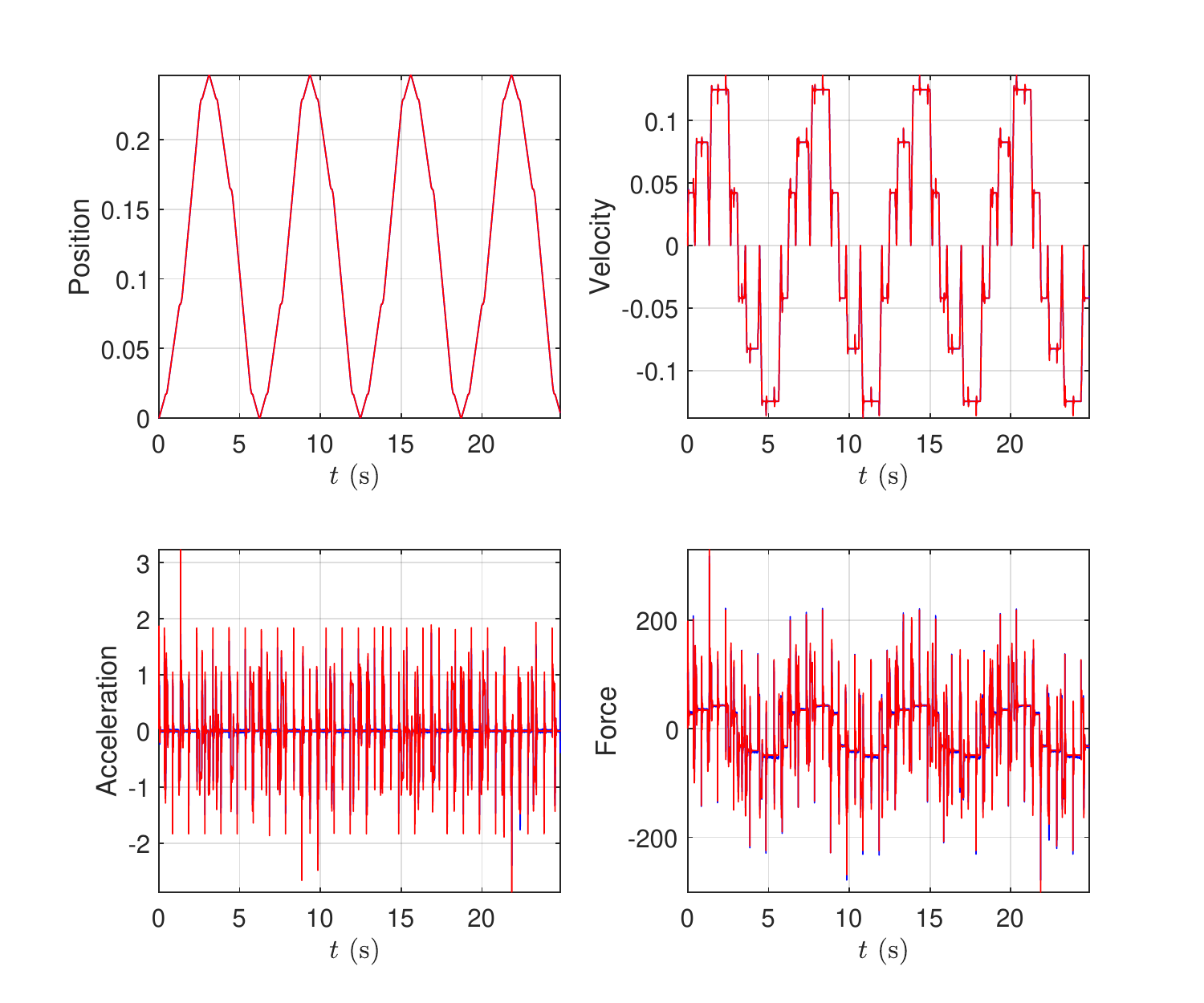}
    \caption{EMPS: Cross-test validation between the simulated and
measured data.}
    \label{fig:EMPS}
\end{figure}

\end{experiment}

\begin{figure}
    \centering
        \includegraphics[scale=0.4]{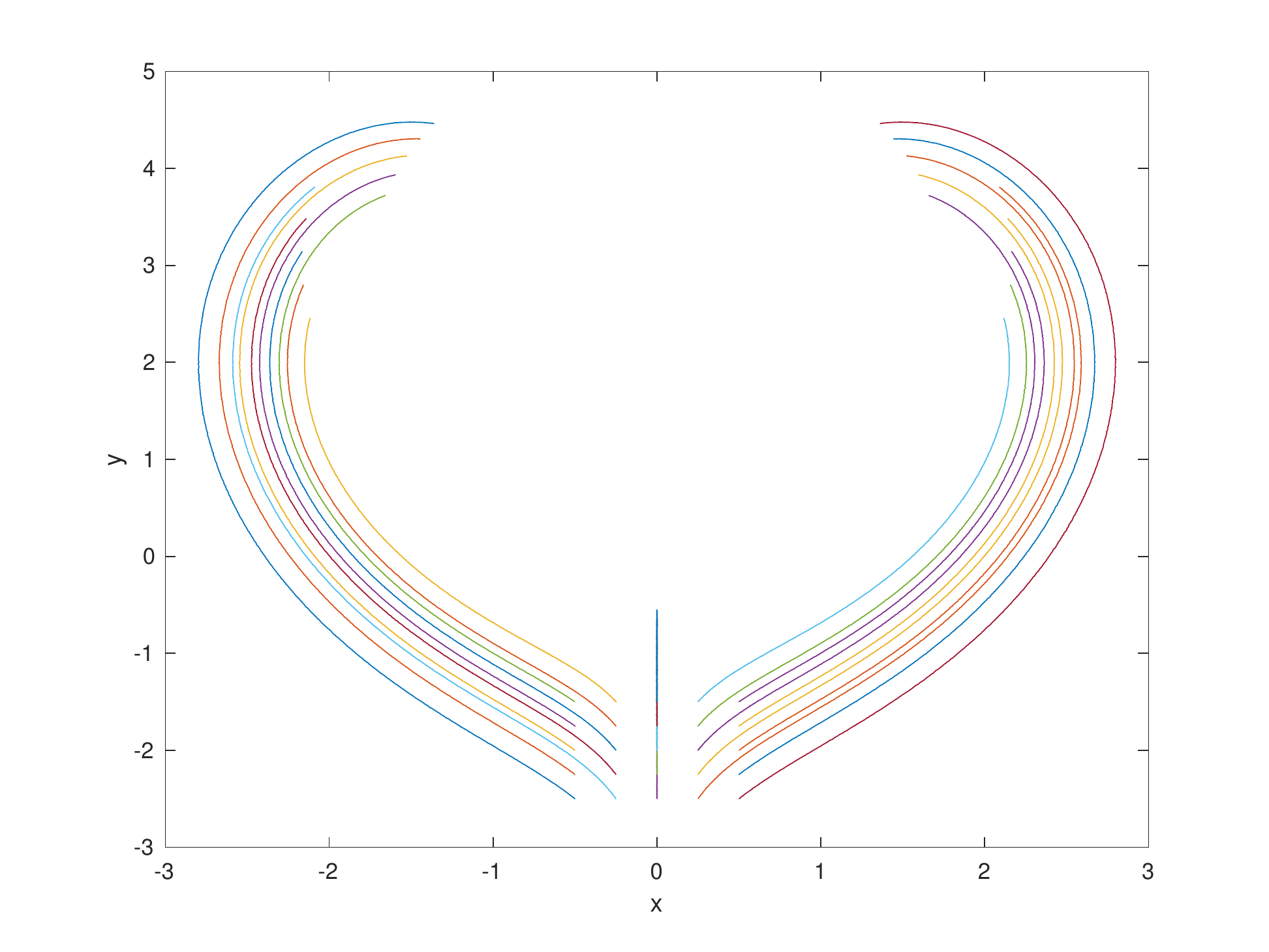}
    \caption{Twenty five trajectories corresponding to System \ref{sys:one}.}
    \label{fig:system1}
\end{figure}

\begin{figure}
    \centering
    \includegraphics[scale=0.4]{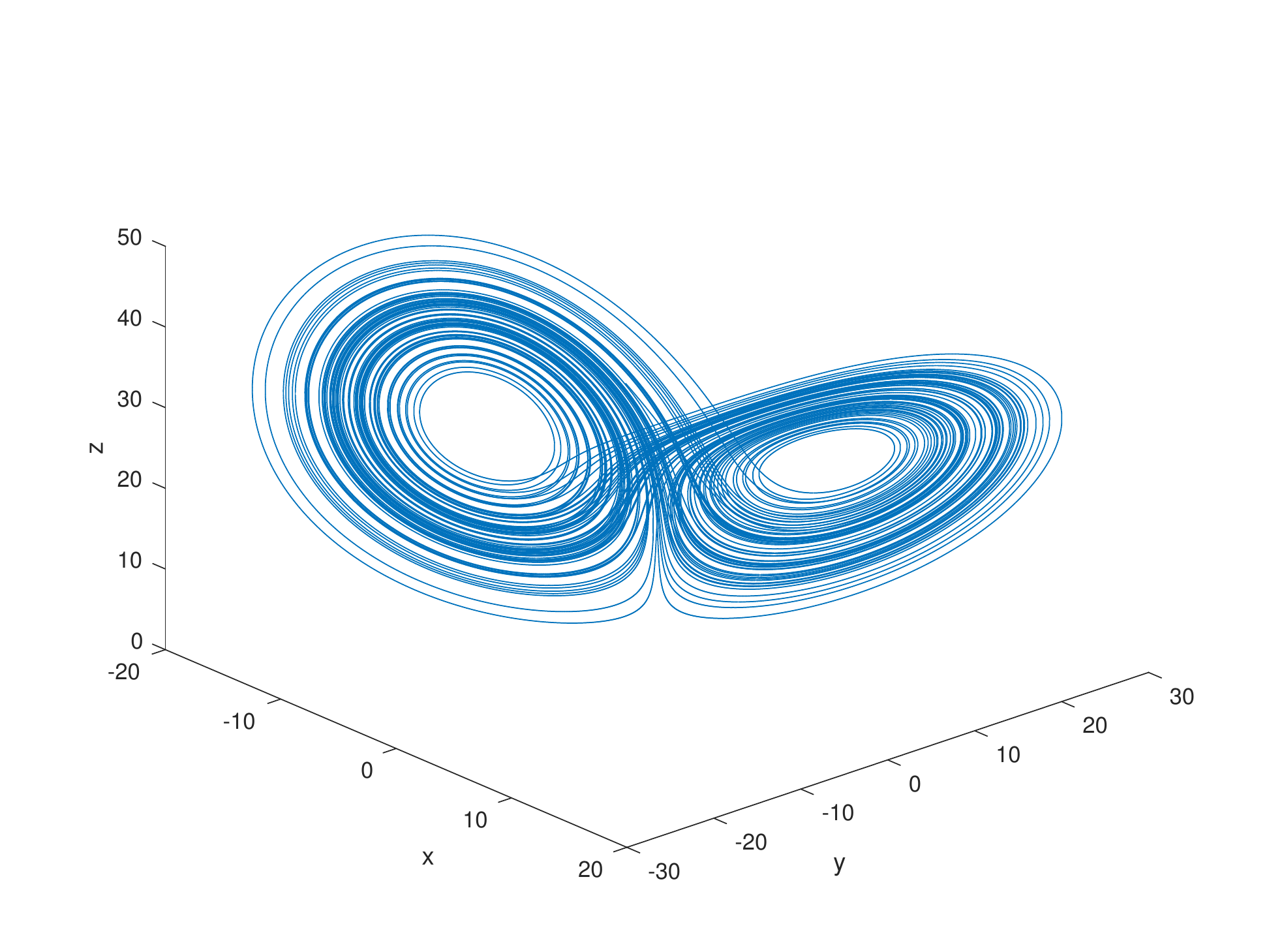}
    \caption{A single trajectory for the three dimensional Lorenz system given in Example 2. }
    \label{fig:system2}
\end{figure}

\begin{figure}
    \centering
    \includegraphics[scale=0.4]{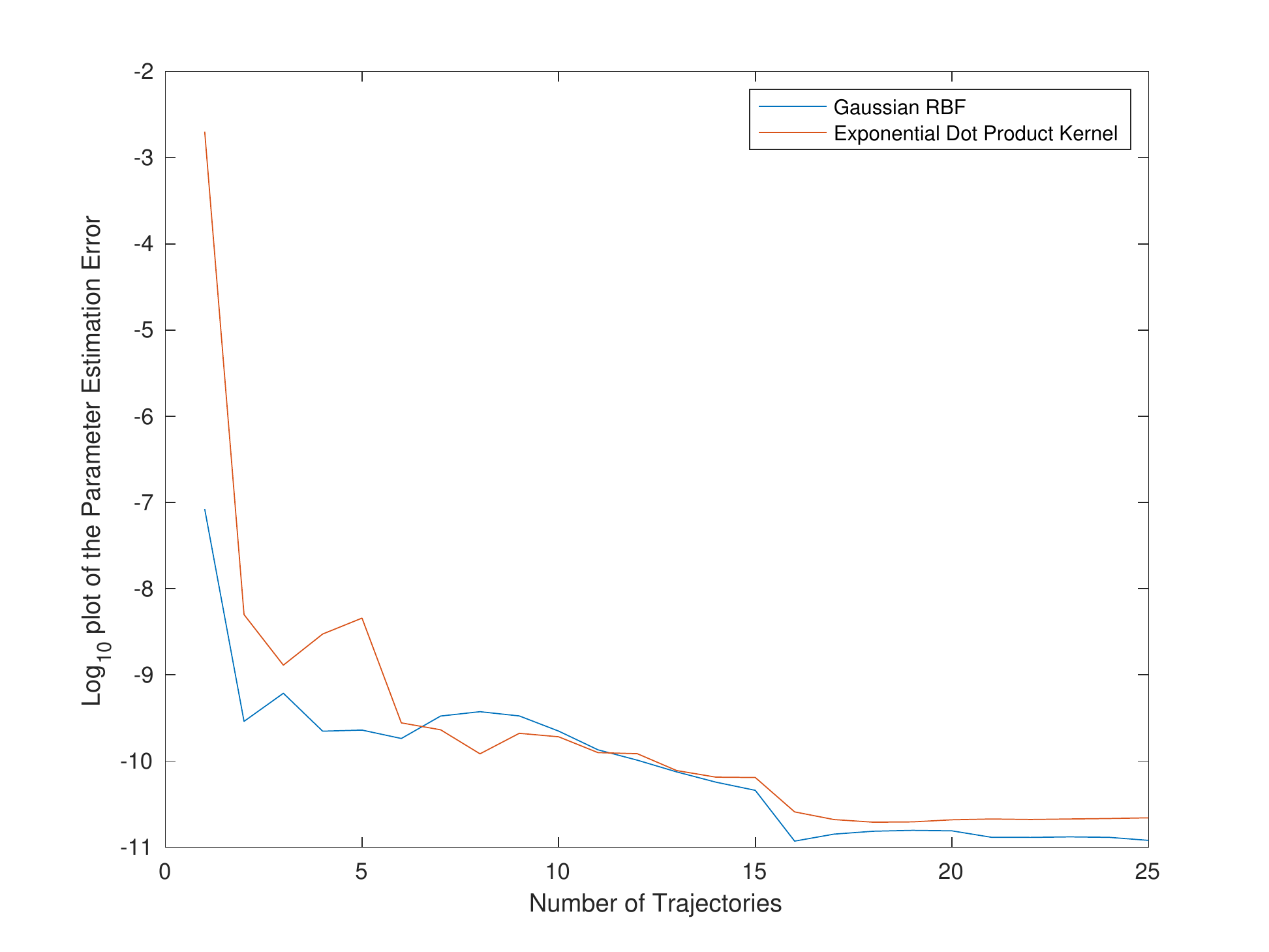}
    \caption{A log-plot of the parameter estimation error, $\|\theta-\hat\theta\|_2$, for System \ref{sys:one} committed by the system identification method in Section \ref{sec:systemid} as determined by the number of trajectories utilized by the method. It may be observed that an accurate estimate of $\theta$ is established using a single trajectory. However, the inclusion of additional data dramatically improves the parameter estimation error. The two graphs represent the results for two different kernel functions, with a slight advantage exhibited by the Gaussian RBF.}
    \label{fig:varynumber}
\end{figure}

\begin{figure}
    \centering
    \includegraphics[scale=0.4]{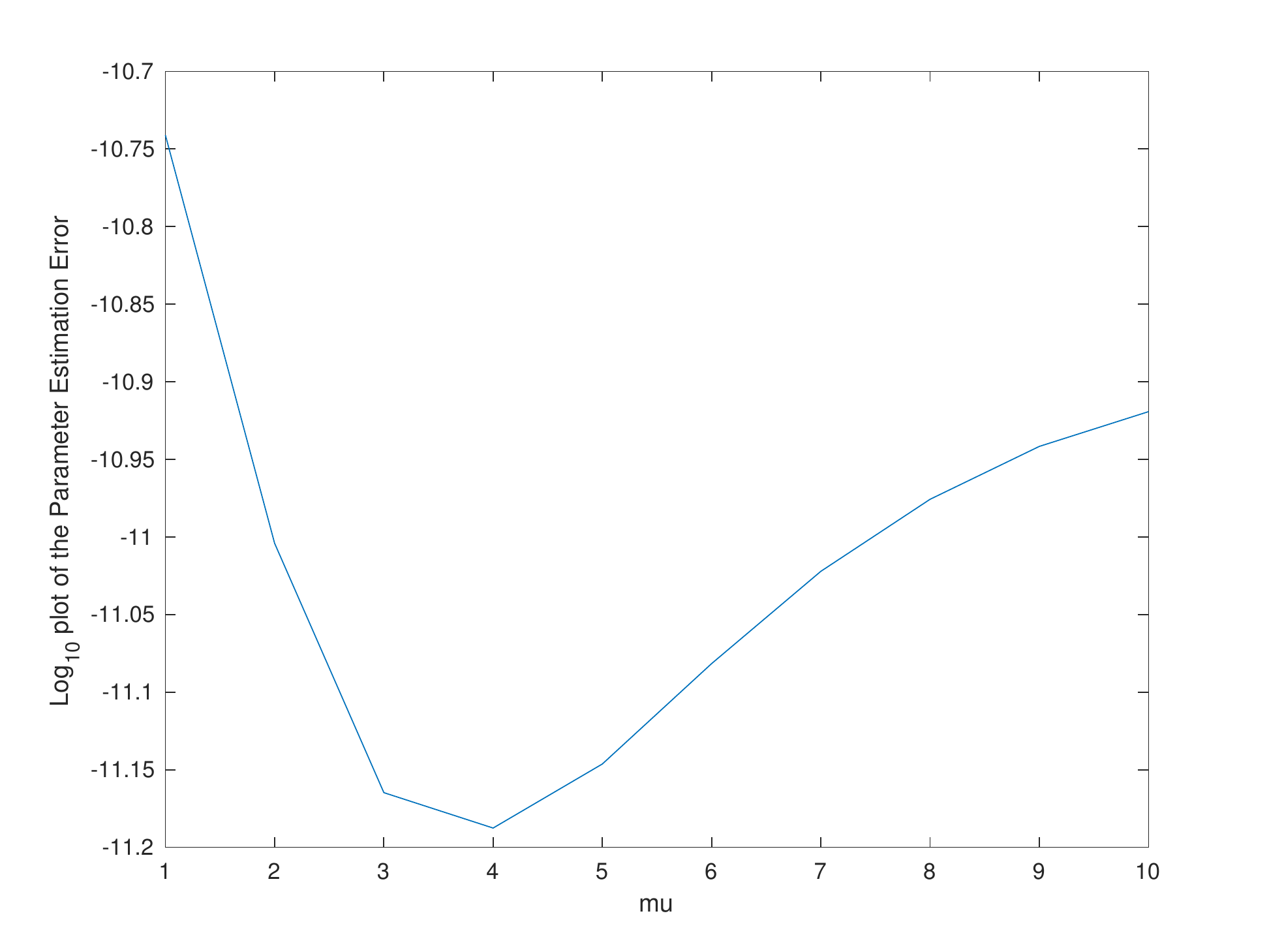}
    \caption{A log-plot of the parameter estimation error versus adjustments in the kernel width $\mu$ for the Gaussian RBF applied to System \ref{sys:one}. Note that the error is small for any selection of $\mu$, but reaches its minimum near $\mu=4$. Additionally note that $4$ is approximately the radius of the workspace. One reason for the loss of accuracy at larger $\mu$ values could be poorer conditioning of the matrix $\bold A$ for large $\mu$.}
    \label{fig:varymu}
\end{figure}

\begin{table}[ht]
\centering
\begin{tabular}{ccrrrr}
{\ul \textbf{Monomial}}                      & {\ul \textbf{Dim}}          & {\ul \textbf{No Noise}}                           & {\ul \textbf{Noise}}                               & {\ul \textbf{Moving Average}}                  & {\ul \textbf{Target}}      \\
$x_1^0x_2^0$                                 & 1                                 & 8.882e-16                               & 8.051e-3                                & -2.093e-3                              & 0                                  \\
{\color{black} \textbf{$x_1^1x_2^0$}} & {\color{black} \textbf{1}} & {\color{black} \textbf{2.000e+0}}   & {\color{black} \textbf{2.000e+0}}   & {\color{black} \textbf{2.000e+0}}   & {\color{black} \textbf{2}}  \\
$x_1^2x_2^0$                                 & 1                                 & 8.882e-16                               & -3.840e-3                               & 1.346e-3                                & 0                                  \\
$x_1^0x_2^1$                                 & 1                                 & -1.554e-15                              & 3.968e-3                                & -1.523e-3                               & 0                                  \\
{\color{black} \textbf{$x_1^1x_2^1$}} & {\color{black} \textbf{1}} & {\color{black} \textbf{-1.000e+0}} & {\color{black} \textbf{-9.988e-1}} & {\color{black} \textbf{-9.994e-1}} & {\color{black} \textbf{-1}} \\
$x_1^0x_2^2$                                 & 1                                 & -6.939e-17                              & -2.173e-4                              & -3.066e-4                              & 0                                  \\
$x_1^0x_2^0$                                 & 2                                 & -8.691e-12                              & -7.179e-3                               & -1.363e-3                               & 0                                  \\
$x_1^1x_2^0$                                 & 2                                 & 0                                                  & 1.471e-3                              & -2.271e-4                          & 0                                  \\
{\color{black} \textbf{$x_1^2x_2^0$}} & {\color{black} \textbf{2}} & {\color{black} \textbf{2.000e+0}}   & {\color{black} \textbf{2.003e+0}}   & {\color{black} \textbf{2.001e+0}}   & {\color{black} \textbf{2}}  \\
{\color{black} \textbf{$x_1^0x_2^1$}} & {\color{black} \textbf{2}} & {\color{black} \textbf{-1.000e+0}}   & {\color{black} \textbf{-1.007e+0}}  & {\color{black} \textbf{-1.001e+0}}   & {\color{black} \textbf{-1}} \\
$x_1^1x_2^1$                                 & 2                                 & -8.327e-17                              & 1.834e-4                               & 9.210e-5                               & 0                                  \\
$x_1^0x_2^2$                                 & 2                                 & 2.652e-12                               & 8.280e-4                               & -2.650e-4                              & 0             \\
\bf \color{black} Max Error & & \bf \color{black} 8.691e-12 &   \bf \color{black} 8.051e-3 &  \bf \color{black} 2.093e-3   &                
\end{tabular}
\caption{\label{tab:experiments}This table presents the results of the nonlinear system identification method applied to the trajectories presented in Figure \ref{fig:system1}. The target parameters are listed in the last column, and the columns ``No Noise,'' ``Noise,'' and ``Moving Average'' show the obtained parameters from their respective experiments. The ``Monomial'' column lists the specific basis function that the parameter of that row is tied to, and ``Dim'' expresses which dimension that particular basis function is contributing to. Note that the bolded rows correspond to the non-zero target values.
The presented results demonstrate that even in the case of unfiltered noise, the nonlinear system identification method of Aim 1 obtains parameter estimates while committing an error of a most $8.051e-3$.}
\end{table}

\begin{figure}
    \centering
    \includegraphics[scale=0.4]{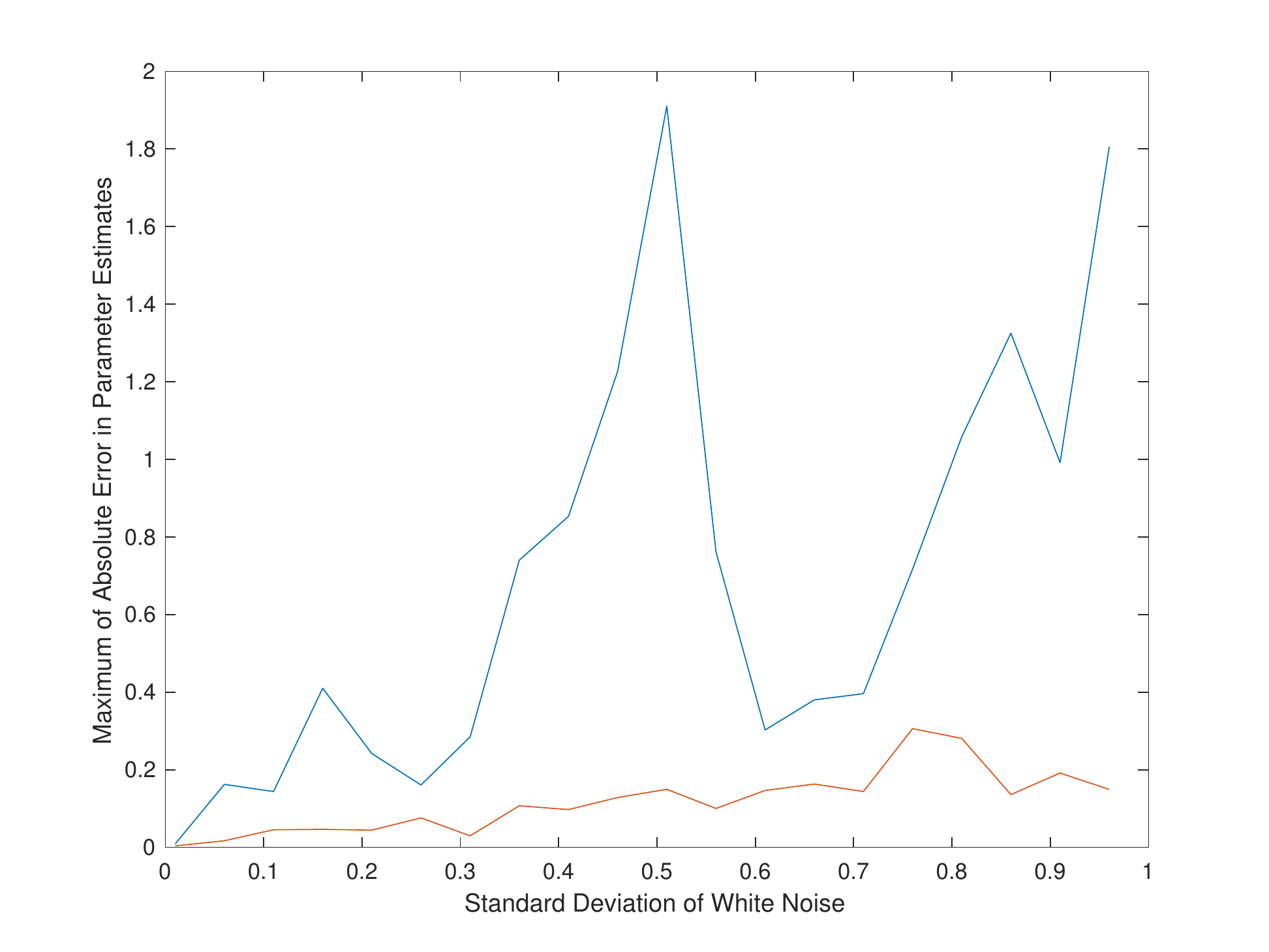}
    \caption{This figure shows the impact of noise on the maximum of the absolute errors of the parameters for System \ref{sys:one}. The blue plot corresponds to parameters obtained using the noisy data, and the red corresponds to those parameters obtained after a $20$ point moving average filter. It can be seen that in combination with a moving average filter, the errors of the parameter estimates remain low as the standard deviation of the noise increases.}
    \label{fig:noise-graph}
\end{figure}

\section{Discussion}\label{sec:discussion}

It may be observed through the numerical experiments performed in Section \ref{sec:numerical} that the system identification method of  Section \ref{sec:systemid} is effective at identifying the parameters for nonlinear systems. In particular, for System \ref{sys:one} parameter estimation errors were as low as $10^{-11}$ and for System \ref{sys:lorenz} the parameter estimation errors were as low as $10^{-5}$. The systems given in Section \ref{sec:numerical} are of two and three dimensions, and the dynamics are nonlinear. The basis functions utilized to represent the unknown dynamics are monomials of degree up to two with appropriate dimensionality. For example, for a three dimensional system the cardinality of the basis of monomials of degree up to two is 30 when accounting for each dimension (i.e. there is a copy of the 10 monomial basis vectors for each dimension). The actual dynamics in each case use only a handful of the basis functions, which results in a sparse representation of the dynamics in the given basis.

The adjustment of several parameters affect the accuracy of the determine parameters, $\theta$. The most obvious impact on the accuracy of the parameters arises through the selection of the kernel function. While theoretically it is established that Liouville operators with polynomial symbols are densely defined over the exponential dot product kernel's native space, the exponential dot product kernel suffers from poor conditioning. This poor conditioning can lead to inaccuracies that appear from numerical uncertainties in the expression of the (left) inverse matrix for $\bold A$ in \eqref{eq:matrix-constraints}. The Gaussian RBF exhibits less conditioning issues than the exponential dot product kernel, especially when a small kernel width is selected. In the case of the Gaussian RBF, the size of the kernel width has an impact on the accuracy of the system identification method as shown in Figure \ref{fig:varymu}. Specifically, occupation kernels corresponding to Gaussian RBFs with smaller kernel widths can distinguish nearby trajectories more effectively than those with larger kernel widths, which leads to better conditioning of $\bold A$ in \eqref{eq:matrix-constraints}. However, it is well known in approximation contexts that larger values of $\mu$ lead to faster convergence \cite{fasshauer2007meshfree}. The minimum error at $\mu=4$ in Figure \ref{fig:varymu} thus strikes a balance between the conditioning of the matrix and the advantages gained from larger $\mu$.

The most significant contribution to errors in the estimation of the parameters is the method of numerical integration performed. The simple example presented in Proposition \ref{prop:quadrature-convergence} gives an estimation of the occupation kernel via a right hand rule method of numerical integration, and while Proposition \ref{prop:quadrature-convergence} provides a proof of concept demonstrating norm convergence to the occupation kernel in question, it is observed in \eqref{eq:sqrth} that this method results in a relatively slow convergence rate. When other methods, such as the trapezoid or Simpson's rule is leveraged for numerical integration, a significant improvement in the performance of the system identification method may be realized as demonstrated in Table \ref{tab:numschemes}. Consequently, the fourth order method of Simpson's rule was utilized for most of the results presented in Section \ref{sec:numerical}.

{\color{black}Given enough segments of the trajectory, the parameters can be estimated using the more straightforward ILS approach in \eqref{eq:ILS}. However, as evidenced by Figure \ref{fig:MonteCarloILSComparison}, the kernel approach developed in this paper can yield better parameter estimates. The improvement is attributed to the ability of the developed framework to extract more information from the same set of trajectories by integrating them against different test functions.}

\begin{figure}
    \centering
    \subfloat[Parameter estimation error resulting from occupation kernel system identification and integral least squares system identification.]{\label{fig:MonteCarloILSComparisonError}\includegraphics[width = 0.48\textwidth]{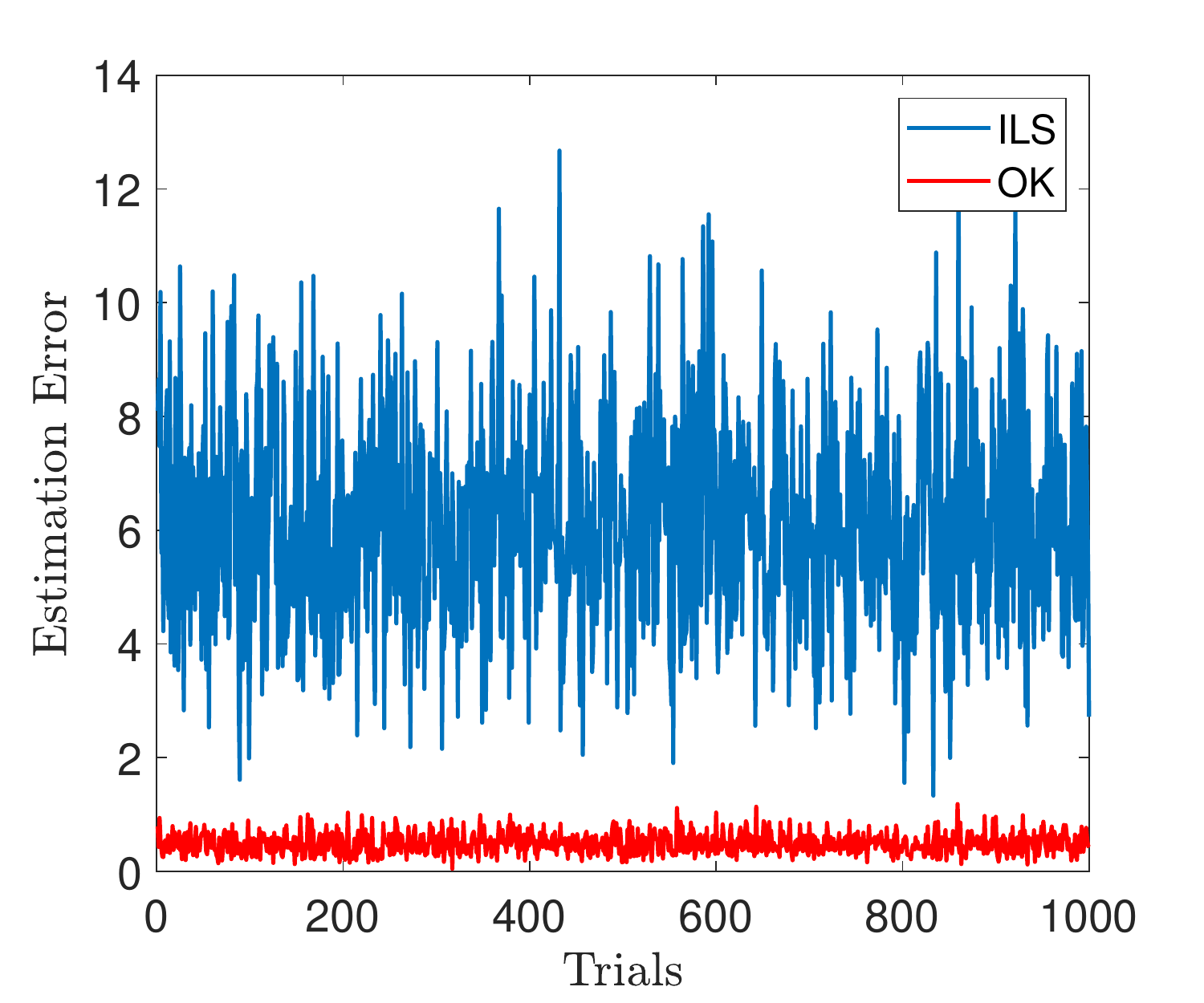}}\hfill
    \subfloat[Condition number of the regression matrix resulting from occupation kernel system identification and integral least squares system identification.]{\label{fig:MonteCarloILSComparisonConditionNumber}\includegraphics[width = 0.48\textwidth]{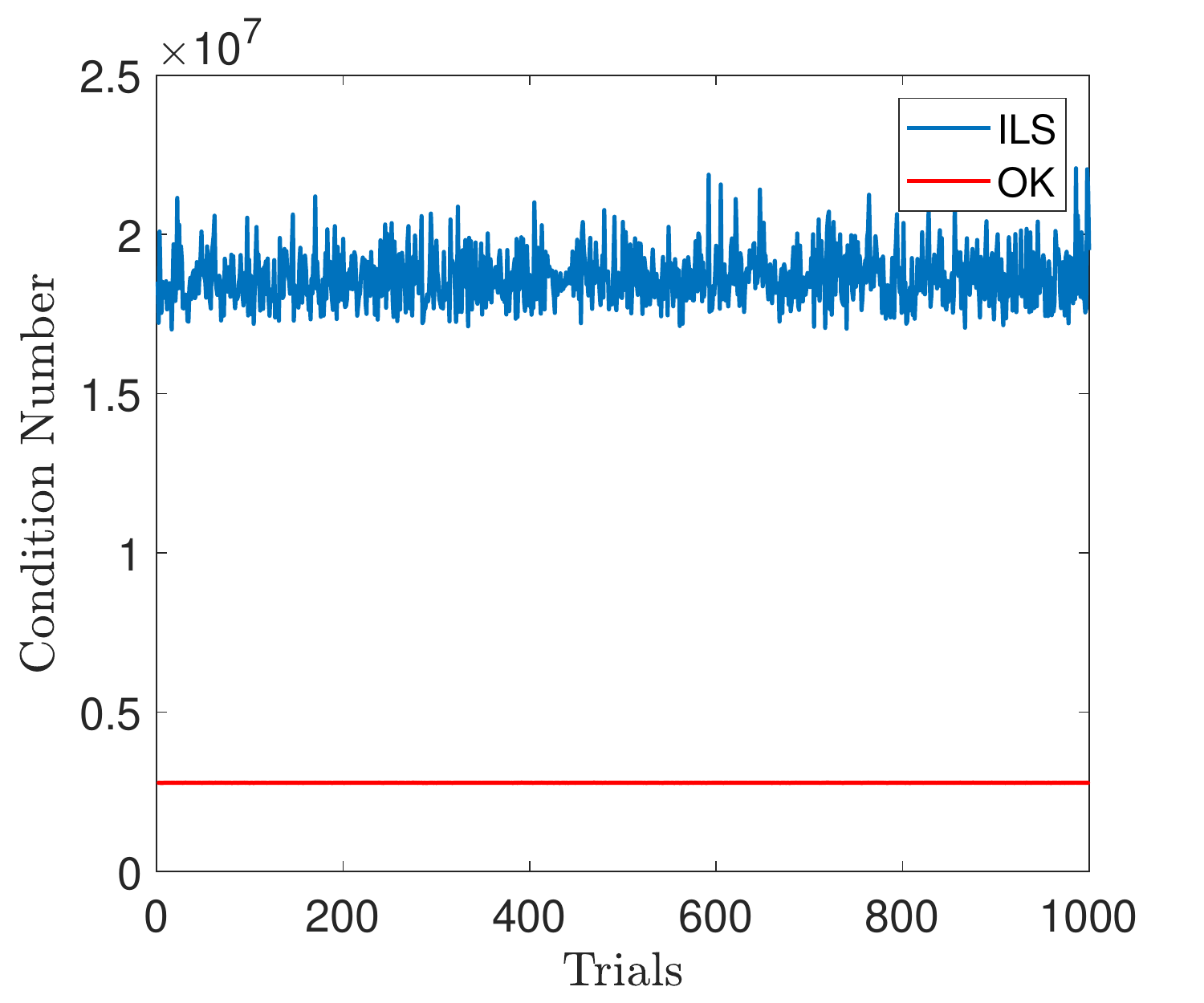}}
    \caption{This figure summarizes the results of a set of 1000 Monte-Carlo trials comparing the ILS formulation in \eqref{eq:ILS} with the developed occupation kernel (OK) system identification method. See Experiment \ref{exp:five} for details.}
    \label{fig:MonteCarloILSComparison}
\end{figure}

Two other factors that contribute to the success of the system identification algorithm of Section \ref{sec:systemid} are the selection of the centers of the kernel functions as well as the number of trajectories. The contribution of the Gaussian RBFs are largest when the centers are distributed over the working space containing the trajectories. That is, if the centers are too far away from the trajectories, the decay of the Gaussian RBFs will lead to near zero row vectors of $\bold A$ in \eqref{eq:matrix-constraints}. For the algorithm in Section \ref{sec:systemid}, each kernel function is evaluated for every trajectory, but this isn't technically necessary and kernel functions that will contribute less or redundant information may be ignored for a specific trajectory.

If only a single trajectory is available from a system, as with the Lorenz example in Section \ref{sec:numerical}, then this trajectory may be segmented to provide more constraints in $\bold A$ of \eqref{eq:matrix-constraints}. It was observed that segmenting the long trajectory of System \ref{sys:lorenz} improved the parameter estimation error as presented in Table \ref{tab:segmented}. This improvement in likely due to the accumulated global error due to numerical time stepping methods in the generation of the trajectory itself. Through segmentation, each segment is treated as a new initial value problem with a smaller time horizon and thus a smaller accumulated global error. Therefore the elements of $\bold A$ in \eqref{eq:matrix-constraints} are closer to the true values of the dynamical system.

Experiment \ref{exp:six} provides a real world test of the system identification method presented in this manuscript. Several features appear in the dynamics of this system that are not present in the other experiments including discontinuities in the acceleration. Table \ref{tab:EMPSexperiments} demonstrates that performance of the present method closely matches that of the benchmark method from \cite{SCC.Janot.Gautier.ea2019}. It should be noted that the poor performance of both methods in estimating the acceleration occurs because of the discontinuities in acceleration, where sharp ``ringing'' phenomena occur at the jumps in acceleration. This occurs as the basis functions used to approximate the dynamics corresponding to acceleration are continuous.

The key difference between existing kernel based approaches and the present approach is that the basis functions are separated from data integration. Whereas typical kernel based approaches leverage the representer theorem to yield an approximation of the dynamics with respect to a linear combination of kernel functions centered at the data points (cf. \cite{SCC.Pillonetto.Dinuzzo.ea2014}). The method presented in this manuscript leverages the form of the occupation kernel to incorporate the trajectory inside a RKHS, and the selection of occupation kernel is largely independent of the selection of basis functions. In the present context, the selection of basis functions is such that the functions should give a densely defined Liouville operator in the selected kernel space to guarantee the soundness of the developed methods. The advantage gained in the use of occupation kernels is that the occupation kernel itself is not as strongly influenced by noisy measurements as the kernel counterparts, since it can be represented as the integral of kernels that have centers along the trajectory.

\section{Conclusion}
{\color{black}In this manuscript a new approach to system identification was developed through the use of Liouville operators and occupation kernels over a RKHS. Liouville operators are densely defined operators whose adjoint contains occupation kernels corresponding to solutions to differential equations within its domain. Hence, a dynamical system may be embedded into a RKHS where methods of numerical analysis, machine learning, and approximation theory affiliated with RKHSs may be brought to bear on problems in dynamical systems theory. Specifically, an inner product on dynamical systems that give rise to densely defined Liouville operators was developed, where the projection of a dynamical system on a linear combination of basis function is realized through the solution of a parameter identification routine that generalized that found in \cite{brunton2016discovering} through a weak formulation using integrals. In this setting, the features of a kernel function become test functions for designing contraints on parameters for system identification.}

The domain of Liouville operators depends nontrivially on the selection of RKHS. It was demonstrated that Liouville operators with polynomial symbols are densely defined over the RKHS corresponding to the exponential dot product kernel function. Moreover, it was demonstrated in the system identification routine that the selection of kernel function may have an effect on the results of parameter estimation.

The system identification method developed in the manuscript was validated on a two dimensional and a three dimensional system through several different experiments designed to evaluate the effects of various integration and RKHS parameters, such as kernel width for the Gaussian RBF, the selection of numerical integration scheme, the selection of kernel, and so on. Through each experiment, accurate estimations of the parameters were achieved. However, it was demonstrated that the largest error source arises through the choice of numerical integration method, where Simpson's rule provided the most accurate results.

\appendix
{\color{black}
\section{Pre-inner Products for Various Kernel Functions}\label{sec:ExplicitInnerproducts}

To give a better presentation of the above framework, this section performs the evaluation of the pre-inner products expressed in \eqref{eq:innerproduct_testfree} for several kernels. This subsection will demonstrate explicit formulas for the exponential dot product kernel, the Gaussian RBF kernel, and the polynomial kernel. In each case this depends on evaluating $\nabla_1 ( \nabla_2 K(x,y) Y_j(y) ) Y_i(x)$ in the Gram matrix. For each of the respective kernels this results in the following
\begin{gather}
    \label{eq:pre-inner-exp-dot}\mu Y_m(y)^T\left( I_n  + \mu  \diag( x_1 y_1, \ldots, x_n y_n ) \right)Y_{m'}(x) \exp\left(\mu x^T y\right)\\
    \label{eq:pre-inner-gauss-rbf}\frac{2}{\mu} Y_m(y)^T \left( I_n   - \frac{2}{\mu}  \diag( (x_1 - y_1)^2, \ldots, (x_n - y_n)^2 ) \right)Y_{m'}(x) \exp\left( -\frac{\| x - y \|_2^2}{\mu}\right)\\
    \label{eq:pre-inner-poly}\frac{d}{\mu}Y_m(y)^T \left( \left(1+\frac{x^Ty}{\mu}\right) I_n + \frac{d-1}{\mu}\diag( x_1 y_1, \ldots, x_n y_n ) \right) Y_{m'}(x) \left(1 + \frac{x^Ty}{\mu}\right)^{d-2}.
\end{gather}
The pre-inner product is completed when each of $x$ and $y$ are replaced with $\gamma(t)$ and $\gamma(\tau)$ respectively and double integration is performed.

These equations generalize to other dot product and RBF kernels in an obvious way. Significantly, if $Y_m$ is pointwise zero in every dimension that $Y_{m'}$ is nonzero (e.g., $Y_m(x) = (x_1^2,0,0)^T$ and $Y_{m'}(x) = (0,x_2x_1,0)^T$), then the pre-inner product is null. Hence, if the bases are expressed as collections of monomials in only one dimension at a time, then the resultant Gram matrix is a block diagonal.

For the right hand side of \eqref{eq:testfree_gram}, it is necessary to compute $\langle f, Y_m \rangle_{\mathcal{F},\gamma} = \langle A_f^* \Gamma_{\gamma}, A_{Y_m}^* \Gamma_{\gamma} \rangle_H$ for each $m$. This is facilitated by noting that $\dot \gamma = f(\gamma)$, and hence $A_{f}^* \Gamma_{\gamma} = K(\cdot,\gamma(T)) - K(\cdot,\gamma(0)).$ For each of the kernels under consideration here, it is thus necessary to compute $\nabla_1 (K(x,\gamma(T)) - K(x,\gamma(0))) Y_m(x)$ as

\begin{gather}
    \left(\gamma(T)^T \exp\left( \mu x^T\gamma(T)\right) - \gamma(0)^T \exp\left( \mu x^T\gamma(0)\right)\right) Y_m(x),\\
    -\frac{2}{\mu}((x-\gamma(T))^T\exp\left(-\frac{\|x-\gamma(T)\|_2^2}{\mu}\right) + (x-\gamma(0))^T\exp\left(-\frac{\|x-\gamma(0)\|_2^2}{\mu}\right)Y_{m}(x)\text,\\
    \text{and }\frac{d}{\mu}\left(\gamma(T)^T\left(1 + \frac{x^T\gamma(T)}{\mu}\right)^{d-1} + \gamma(0)^T\left(1 + \frac{x^T\gamma(0)}{\mu}\right)^{d-1} \right) Y_m(x).
\end{gather}
Replacing $x$ in one of the above with $\gamma(t)$ and integrating over $[0,T]$ gives the $m$-th component of the right hand side of \eqref{eq:testfree_gram}.}

\section{Homotopies and Occupation Kernels}

As part of the purpose of this manuscript is to introduce occupation kernels, this section aims to demonstrate additional continuity results. In particular, a connection between homotopies and occupation kernels is present below.  
\begin{definition}
Let $\gamma_1:[0,T] \rightarrow Y$ and $\gamma_2:[0,T]\rightarrow Y$ be continuous functions. A homotopy between $\gamma_1$ and $\gamma_2$ exists if there is a continuous function $h:[0,T]\times [0,1]\rightarrow Y$ such that $h(t,0)=\gamma_1(t)$ and $h(t,1)=\gamma_2(t)$. Alternatively, a homotopy between the continuous functions $\gamma_1$ and $\gamma_2$ is a family of contiuous functions $h_t:X\rightarrow Y$ such that $h_0(x)=f(x)$, $h_1(x)=g(x)$ and the mapping $(x,t)\rightarrow h_t(x)$ is continuous. 
\end{definition}

\begin{proposition}
Suppose $H$ is a RKHS over a set $X$ consisting of continuous functions and let $\gamma_1(t)$ and $\gamma_2(t)$ be two paths with homotopy $\{\gamma_s(t)\}$. The map 
$s\mapsto \Gamma_{\gamma_s}$
is continuous. 

\end{proposition}
\begin{proof}
As $[0,T] \times [0,1]$ is compact, the map $(t,s) \mapsto \gamma_s(t)$ is uniformly continuous. That is for every $\varepsilon > 0$ there exists a $\delta > 0$ such that whenever $\|(t_1,s_1)-(t_s,s_2)\|_2 < \delta$, $\| \gamma_{s_1}(t_1) - \gamma_{s_2}(t_2) \|_2 < \epsilon.$ Moreover, as $K(\cdot,\cdot)$ is continuous, and the image of $\gamma_{s}(t)$ is compact, the map $(s_1,t,s_2,\tau) \mapsto K(\gamma_{s_1}(t),\gamma_{s_2}(\tau))$ is uniformly continuous.

Fix $\varepsilon > 0$ and select $\delta > 0$ such that 
\begin{gather*}
    |K(\gamma_{s_1}(t),\gamma_{s_1}(\tau)) - K(\gamma_{s_2}(t),\gamma_{s_1}(\tau))| < \frac{\varepsilon}{2T^2} \text{ and}\\
    |K(\gamma_{s_2}(t),\gamma_{s_2}(\tau)) - K(\gamma_{s_2}(t),\gamma_{s_1}(\tau))| < \frac{\varepsilon}{2T^2}
\end{gather*}
whenever $|s_1 - s_2| < \delta$. Select $s_1, s_2$ such that $|s_1 - s_2 < \delta$, then
\begin{gather}\label{eq:homo-equality}
\| \Gamma_{s_1} - \Gamma_{s_2} \|_H^2 =  \langle \Gamma_{s_1}, \Gamma_{s_1} \rangle_H + \langle \Gamma_{s_2}, \Gamma_{s_2} \rangle_H - 2 \langle \Gamma_{s_1}, \Gamma_{s_2} \rangle_H  \nonumber \\
= \int_0^T \int_0^T ( K(\gamma_{s_1}(t),\gamma_{s_1}(\tau)) - K(\gamma_{s_2}(t),\gamma_{s_1}(\tau)) \nonumber \\
+ K(\gamma_{s_2}(t),\gamma_{s_2}(\tau)) -  K(\gamma_{s_2}(t),\gamma_{s_1}(\tau)) ) dt d\tau
\end{gather}
Note that \eqref{eq:homo-equality} is positive and bounded by $\varepsilon$ by construction. Hence, the map $s \mapsto \Gamma_{s}$ is continuous.
\end{proof}


\bibliographystyle{siamplain}

\bibliography{cdc2019refs}

\end{document}